\documentclass{elsarticle}

\usepackage{geometry,graphicx,amssymb,amsthm,amsmath,amsbsy,eucal,amsfonts,mathrsfs,amscd,bm}
\usepackage[all]{xy}
\usepackage{tikz}
\usetikzlibrary{arrows,shapes}
\usetikzlibrary{arrows, decorations.markings,fit}
\usetikzlibrary{calc,3d}
\usepackage{tikz-3dplot}
\usepackage{listings}
\usepackage[fleqn]{mathtools}
\usepackage{framed}
\usepackage{pgfplots}
\usepackage{svg}
\usepackage{subfigure}
\usepackage{epsfig}
\usepackage{colordvi}
\usepackage{multirow}

\usepackage{graphics}
\usepackage{float}
\usepackage{ mathrsfs }
\newenvironment{mathframed}{\framed%
\allowdisplaybreaks
\vspace*{-\abovedisplayskip}\noindent}{%
\vspace*{-\dimexpr\baselineskip+\topsep}\endframed}
 \usepackage{array,multirow,graphicx}
\newcommand*\rot{\rotatebox{90}}

\numberwithin{equation}{section}

\allowdisplaybreaks[4]

\newtheorem{theorem}{Theorem}[section]
\newtheorem{lemma}[theorem]{Lemma}
\newtheorem{assumption}[theorem]{Assumption}
\newtheorem{algorithm}[theorem]{Algorithm}

\newtheorem{definition}[theorem]{Definition}
\newtheorem{remark}[theorem]{Remark}
\newcommand{\nr}[1]{\ensuremath{\left\|{#1}\right\|}}

\newcommand{\Ome}{{\Omega}}

\newcommand{\p}{{\partial}}

\newcommand{\nab}{\nabla}

\title{Improved convergence of the Arrow-Hurwicz iteration for the Navier-Stokes equation via grad-div stabilization and Anderson acceleration}

\author[label0]{Pelin G. Geredeli}
\ead{peling@iastate.edu}

\author[label1]{Leo G. Rebholz}
\ead{rebholz@clemson.edu}

\author[label1]{Duygu Vargun}
\ead{dvargun@clemson.edu}

\address[label1]{Department of Mathematical Sciences, Clemson University, Clemson, SC 29634, USA}
\address[label0]{Department of Mathematics, Iowa State University, Ames IA, 50011, USA}

\author[label0]{Ahmed Zytoon}
\ead{AMZ56@pitt.edu}

\begin{document}

\begin{abstract}  
We consider two modifications of the Arrow-Hurwicz (AH) iteration for solving the incompressible steady Navier-Stokes equations for the 
purpose of accelerating the algorithm: grad-div stabilization, and Anderson acceleration.  AH is a classical iteration for general
saddle point linear systems and it was later extended to Navier-Stokes iterations in the 1970's which has recently come under study again.
We apply recently developed ideas for grad-div stabilization and divergence-free finite element methods along with Anderson acceleration 
of fixed point iterations to AH in order to improve its convergence.  Analytical and numerical results show that each of these methods
improves AH convergence, but the combination of them yields an efficient and effective method that is competitive with more commonly used solvers.
\end{abstract}


\maketitle

\section{Introduction}

We consider in this paper solving the incompressible steady Navier-Stokes equations (NSE) with the Arrow-Hurwicz (AH) iteration.  The steady NSE
defined on a domain $\Omega \subset \mathbb{R}^d$ ($d$=2 or 3) are given by
\begin{subequations}
\label{eqn:3DProblem}
\begin{alignat}{2}
\label{eqn:3DProblem1}
-\nu \Delta {u} + (u\cdot\nab)u + {\nab} p & = f \qquad &&\text{in }\Omega,\\
\label{eqn:3DProblem2}
\nab\cdot u & =0 \qquad &&\text{in } \Omega,\\
\label{eqn:3DProblem3}
u & = 0\qquad &&\text{on }\p \Omega,
\end{alignat}
\end{subequations}
where $u$ and $p$ represent the unknown velocity and pressure, $f$ a given forcing, and $\nu$ the kinematic viscosity which is 
inversely proportional to the Reynolds number $Re$.  {\color{black} Extension of this work to one time step in a temporal discretization of the time dependent NSE
is straight-forward.}

Among various novel iterative methods for solving saddle point systems, the AH algorithm for the steady NSE was seemingly first studied by Temam in 1977 in \cite{T77}, and also {\color{black} more recently} in \cite{CHS17}.  The AH iteration is given with the following decoupled equations:
\begin{align*}
-\frac{1}{\rho} \Delta (u^{m+1} - u^m) - \nu \Delta u^m + u^m \cdot\nabla u^{m+1} + \nabla p^m & = f, \\
\alpha( p^{m+1} - p^m) + \rho \nabla \cdot u^{m+1} & = 0,
\end{align*}
with $\rho$ and $\alpha$ being user determined parameters.  We note that if $\rho=\nu^{-1}$ then AH is exactly the modified Uzawa algorithm from \cite{CHS15}.  This iteration is interesting because it is efficient since the two equations decouple, with the second
equation being simple and the first equation requiring a typical convection diffusion solve{\color{black}r} where one controls the diffusion coefficient (in each iteration) with $\rho$.  Hence from
an implementation (i.e. linear algebraic) point of view, the cost of one AH iteration is very cheap compared to that of a typical Picard or Newton iteration which needs to resolve a saddle point system.  However, a serious drawback of the AH {\color{black}method} is that its convergence properties are not particularly good and even though each iteration is cheap, the total number of iterations can be very large.
The purpose of this paper is to improve the AH algorithm so that it is a competitive and even attractive method for efficient computing of accurate {\color{black}steady} NSE solutions.  We enhance the AH {\color{black}method} with two recently developed ideas, one from computational fluid dynamics (grad-div stabilization) and the other from nonlinear solver theory (Anderson acceleration). {\color{black}Indeed, we show that the combination of these two improvements theoretically and computationally yields that AH method can become a very good solver.}

The classical AH iteration developed in 1958 by Arrow and Hurwicz \cite{AH58} is a stationary iterative method to solve saddle point linear systems.
As noted in \cite{BGL05}, this linear algebraic AH iteration can be regarded as an inexpensive alternative to the (linear algebraic) Uzawa method \cite{U58} whenever solves with 
the matrix arising from the convection-diffusion operators are expensive.  Temam seems to be the first to export the AH iteration ideas to Galerkin methods for solving the
steady NSE, and was able to prove convergence (although without a rate) under certain choices of parameters.  The iteration was (finally) proven to be contractive in 2017 \cite{CHS17}, where it was shown  to be linearly convergent under very small data and certain choices of parameters.  While this was a big step forward for AH and the convergent rate was proven less than 1, the exact rate was not easy to decipher and in practice could be very close {\color{black}to} 1.  The numerical tests in \cite{CHS17} for some relatively easy problem revealed {\color{black}that} the AH {\color{black}method} could be made to converge with good parameter choices, but the number of iterations could be very large (e.g. over 700 iterations {\color{black}with}  $Re=100$ {\color{black}for a} 2d driven cavity problem, and over 10,000 for a 2d steady flow past a cylinder).  While iterations of AH would likely be 5-20 times cheaper than one iteration of usual Picard (e.g. if Krylov solvers with preconditioners such as those in \cite{HR13,BB12,benzi} to solve the saddle point linear systems at each iteration), such high iteration counts still make AH uncompetitive.

Herein we aim to improve {\color{black}the convergence properties of the AH method} by enhancing it with two techniques. The first is the addition of grad-div stabilization {\color{black} which refers to} consistent penalization term that adds $0=-\gamma \nabla (\nabla \cdot u)$ to the NSE momentum equation before discretizing, where $\gamma>0$ is a user defined parameter (how large it should be depends on many factors, see e.g. \cite{JLMNR17}).  It was first proposed by Hughes and Franca in 1988 \cite{FH88}, and has been shown to improve accuracy of finite element approximations \cite{OR04}, improve saddle point linear solvers \cite{OR04,benzi,HR13}, and help various NSE nonlinear iterative solvers converge faster, e.g. \cite{RX15,RVX18}.  The grad-div stabilized AH iteration takes the form
\begin{align*}
-\frac{1}{\rho} \Delta (u^{m+1} - u^m) - \nu \Delta u^m + u^m \cdot\nabla u^{m+1} - \gamma \nabla (\nabla \cdot u^{m+1}) + \nabla p^m & = f, \\
\alpha( p^{m+1} - p^m) + \rho \nabla \cdot u^{m+1} & = 0.
\end{align*}
We show {\color{black}that} the existing convergence theory can {\color{black}dramatically} be improved with the use of grad-div stabilization theory.  Moreover, 
we show that under a certain choice of parameters and in a particular (but commonly used) discrete setting, grad-div stabilized AH {\color{black}method} is equivalent to the classical iterated penalty Picard
iteration.  By establishing this connection, we are able to bring to bear the long established theory for this classical iteration to the AH setting, which establishes a linear convergence rate close to that of Picard for sufficiently large $\gamma$.  

The second enhancement we provide to the AH {\color{black}method} is Anderson acceleration (AA).  AA is an extrapolation technique used to improve convergence of fixed point iterations.  It was first developed in 1965 by D.G. Anderson \cite{Anderson65}, and its use has exploded in the last decade after the paper of Walker and Ni in 2011 showed how effective AA can be on a wide range of problems \cite{WaNi11}.  It has recently been used to improve convergence and robustness of solvers for various types of flow problems \cite{LWWY12,PRX19,PRX21}, geometry optimization \cite{PDZGQL18},  radiation diffusion and nuclear physics \cite{AJW17,TKSHCP15}, molecular interaction \cite{SM11}, and many others e.g. \cite{WaNi11,K18,LW16,LWWY12,FZB20,WSB21,HS16}.  In \cite{RVX21}, AA was shown to significantly improve the convergence and robustness for the IPP method for the NSE and allow for a much wider range of penalty parameter choices.  Given the success AA has had in improving other types of nonlinear iterations for the NSE, applying it to the AH {\color{black}method} seems a natural next step.  Moreover, due to its dramatic improvement of the IPP method in \cite{RVX21} and our showing the strong connection of grad-div stabilized AH {\color{black}method} to IPP {\color{black}method}, applying AA to grad-div stabilized AH seems an optimal combination to improve AH convergence behavior.  A general convergence framework was developed for AA in \cite{EPRX20} and then sharpened in \cite{PR21} which allows for theoretical justification of improved linear convergence from AA, if the associated fixed point function satisfies sufficient smoothness properties.  We will set up the AH {\color{black}iteration} as a fixed point problem and prove that its fixed point operator satisfies the assumptions needed to apply the AA convergence theory.  Furthermore, extensive computations of AH with AA are performed, and AA is observed to provide a dramatic improvement in convergence behavior, with the best convergence coming from combining AA with grad-div stabilization.

This article is arranged as follows.  Section 2 provides the necessary notation and mathematical preliminaries used throughout the paper.  In section 3 we consider the theoretical improvement provided by grad-div stabilization, while in section 4 we show how the fixed point operator associated with the AH iteration allows for the AA theory from \cite{PR21} to be applied.  Finally, in section 5, we give results of several numerical tests that show AH {\color{black}method} enhanced with AA and grad-div stabilization can be a very effective nonlinear solver for the steady NSE.

\section{Preliminaries}\label{notation}

In this section we provide some mathematical preliminaries and notation that will be {\color{black}used} throughout the paper.  We begin by defining the following function spaces on a domain $\Omega$ that either has smooth boundary or is a convex polygon:
\begin{align*}
L^2( \Omega)&:=\{w:\Omega \mapsto \mathbb{R}:\ \|w\|_{{\it L}^2(\Omega)}:= \left(\int_{\Omega}\ {|w|^2}\,dx\right)^{1/2}<\infty \},\\
H^m(\Omega)&:=\{w:\Omega \mapsto \mathbb{R}:\ \|w\|_{{\it H}^m(\Omega)}:= \left(\sum_{|\beta| \leq m} \|{\it D}^\beta w\|_{{\it L}^2( \Omega)}^2\right)^{1/2}<\infty \}.
\end{align*}
{\color{black}Throughout this paper,  $(.,.) $ and $\| \cdot \|$ denote} the inner product and norm on $L^2(\Ome)$, respectively.  All other norms will be denoted with subscripts.  Also, we define the following natural spaces for NSE:
\begin{align*}
Q&:=\{w\in {\it L}^2( \Omega):\ \int_{\Omega}\ {w}\,dx =0  \},\\
X &:=\{{w}\in{\it H}^1(\Omega):\  {w}|_{\partial\Omega}=0   \},
\end{align*}
{\color{black}and we} do not distinguish vector and scalar valued spaces, as it will be clear from context.

The skew-symmetric trilinear form $b^*: X\times X \times X \rightarrow \mathbb{R}$ is defined by
\begin{equation}
b^*(u, v ,w) = \frac{1}{2}(((u\cdot\nab) v ,w) - (u\cdot\nab)w, v )),
\end{equation}
{\color{black}and it can easily be observed that} 
\begin{equation}
b^*(u, v ,v)=0. \label{skew1}
\end{equation}
We will {\color{black}also} utilize the well known bound resulting from H\"older's and Sobolev inequalties \cite{Laytonbook}:
\begin{equation}
\left| b^*(u, v ,w) \right| \le M \| \nabla u \| \| \nabla v \| \| \nabla w \|, \label{skew2}
\end{equation}
where $M$ is a constant depending only on $\Omega$.


\subsection{Finite element preliminaries}

Let $X_h \times Q_h \subset X \times Q$
be conforming and finite dimensional finite element spaces for the velocity and pressure.  Then
a finite element method for \eqref{eqn:3DProblem}, based on the standard velocity-pressure formulation and equipped with grad-div stabilization
seeks $(u,p)\in X_h\times Q_h$
such that $\forall ( v ,q)\in X_h\times Q_h$ we have
\begin{subequations}\label{Eq:FEM_VP}
\begin{alignat}{2}
\nu(\nabla{u},\nabla{ v }) + \gamma(\nab\cdot u,\nab\cdot v ) + b^*(u,u, v ) - (\nab\cdot v , p)&=({f}, v )\\
 (\nab\cdot u,q)&=0.
\end{alignat}
\end{subequations}
Here, $\gamma\ge 0$ is referred to as the grad-div stabilization parameter.  The discrete problem \eqref{Eq:FEM_VP} is well-posed 
if the pair satisfies the inf-sup condition
\[
\sup_{0\neq  v \in X_h} \frac{(\nab\cdot  v ,q)}{\|\nabla  v \|}\ge \beta \|q\|\qquad \forall q\in Q_h,
\]
for some $\beta>0$, and the small data condition $\kappa:= M\nu^{-2} \| f \|{-1}<1$ {\color{black}holds.}  The small data condition is needed for uniqueness (although precisely how sharp it is remains an open question), although existence and boundedness can be proven for any given data.  Common choices that satisfy the inf-sup condition with $\beta$ independent of $h$, and the ones we make herein are $X_h \times Q_h = P_k(\tau_h)\cap X \times P_{k-1}(\tau_h) \cap C^0(\Omega)$ Taylor Hood elements (with $\tau_h$ representing a regular conforming mesh of $\Omega$) and $X_h \times Q_h = P_k(\tau_h)\cap X \times P_{k-1}(\tau_h) \cap Q$ Scott-Vogelius elements with appropriate $k$ and mesh structure (see e.g. \cite{GS19,JLMNR17} for more details).

We will utilize the bound
\begin{equation}
\| \nabla u \| \le \nu^{-1} \| f\|_{-1}, \label{nsebound}
\end{equation}
which is proven in \cite{GR86,Laytonbook} and can be easily deduced from  \eqref{Eq:FEM_VP}.

\subsection{The Arrow-Hurwicz method}

We recall the Arrow-Hurwicz (AH) method from \cite{CHS17} for steady Navier-Stokes equations. The method is given in \cite{T77} with a slight change of parameter variables.

\begin{mathframed}
\begin{algorithm} \label{Alg.1}
Let $\rho , \alpha > 0$ be user selected parameters, then
\begin{enumerate}
\item Let $u^0\in X_h$ and $p^0\in Q_h$ be the solution of the mixed formulation: $\forall ({ v }, q)\in X_h\times Q_h$ : 

\begin{subequations}\label{Eq:CTSS}
\begin{alignat}{2}
(\nabla{u^0},\nabla{ v }) - (\nab\cdot v , p^0)&=({f}, v )\\
(\nab\cdot u^0,q)&=0,
\end{alignat}
\end{subequations}

\item For $m \geq 0 $, we define $u^{m+1}\in X_h$ to be the solution of the following variational\\ formulation : $\forall  v \in X_h$, $u^{m+1}$ satisfies

\begin{equation}\label{AHV}
\frac{1}{\rho}(\nabla(u^{m+1} - u^{m}),\nabla v ) +\nu(\nabla u^{m},\nabla v ) +  b(u^m;u^{m+1}, v ) -       (\nab\cdot v , p^m)=({f}, v ),
\end{equation}

and $p^{m+1}\in Q_h$ to be the solution of the following variational formulation : $\forall q\in Q_h$, $p^{m+1}$ satisfies

\begin{equation}\label{AHP}
\alpha((p^{m+1} - p^{m}),q) + \rho(\nab\cdot u^{m+1}, q)=0.
\end{equation}\\

\end{enumerate}
\end{algorithm}
\end{mathframed}

\begin{remark}
The well-posedness and convergence of the AH scheme was shown in \cite{T77}, under some assumptions on the parameter choices for $\rho$ and $\alpha$.  Under similar choices
and additional data restrictions beyond the small data condition, the AH method was shown to be contractive in \cite{CHS17}.  While contractive, the linear convergence rate is rather hard to decipher from the analysis, which is rather technical (although still an important step forward).  Indeed the rate could be very close to 1, and in the computations with the AH {\color{black}method} in \cite{CHS17}, it appears {\color{black}that} it often is.
\end{remark}

\section{Convergence analysis of a grad-div stabilized AH method}

In this section we consider the following grad-div stabilized AH method, and will show that it has improved convergence properties over the usual AH method.

\begin{mathframed}
\begin{algorithm} \label{Alg.3}
Let $\rho , \alpha > 0$ be parameters, then
\begin{enumerate}
\item Let $u^0\in X_h$ and $p^0\in{\color{black}{Q}_h}$ be the solution of the Stokes problem \eqref{Eq:CTSS}.

\item For $m \geq 0 $, we define $u^{m+1}\in X_h$ to be the solution of: $\forall  v \in X_h$, $u^{m+1}$ satisfies
\begin{multline}\label{ah1}
\frac{1}{\rho}(\nabla(u^{m+1} - u^{m}),\nabla v ) +\nu(\nabla u^{m},\nabla v ) +  b(u^m;u^{m+1}, v )\\
 + \gamma (\nab\cdot u^{m+1}, \nab\cdot v ) -  (\nab\cdot v , p^m)=({f}, v )
\end{multline}
where the grad-div parameter $\gamma>0$ is a user selected parameter, and $p^{m+1}\in Q_h$ to be the solution of: $\forall q\in Q_h$, $p^{m+1}$ satisfies

\begin{equation}\label{ah2}
\alpha((p^{m+1} - p^{m}),q) + \rho(\nabla\cdot u^{m+1}, q)=0.
\end{equation}\\

\end{enumerate}
\end{algorithm}
\end{mathframed}

We show in this section how grad-div stabiliz{\color{black}ation} can provide improved convergence for AH {\color{black}}.  First we show linear convergence through a connection to the classical iterated Picard penalty method, and then we show the classical convergence analysis with the grad-div term included.  Throughout this section $u$ and $p$ represent the solution of \eqref{Eq:FEM_VP}, and we assume the small data condition $\kappa<1$ {\color{black}holds}.

\subsection{Linear convergence of the grad-div stabilized AH {\color{black}iteration} via a connection to the iterated Picard penalty method}

In this section we show that with certain discretizations, linear convergence for the grad-div stabilized AH {\color{black}iteration} can be established under a small data condition.  In particular, we consider
the case of velocity and pressure spaces satisfying both the inf-sup stability condition and $\nabla \cdot X_h=Q_h$.  For example, $(P_2,P_1^{disc})$ Scott-Vogelius elements on Alfeld splits  satisfy this property \cite{JLMNR17}.  We prove that with the right parameters, linear convergence {\color{black}that is} equivalent to that of the iterated {\color{black} Picard penalty (IPP)} method is achieved (which in practice is close to that of the well known Picard method \cite{GR86,C93,HS18,RVX21}).

The IPP method for the steady NSE is a classical method that has been well studied and extensively used \cite{G89It, HS18, RX15,C93}, and is defined {\color{black}\cite{C93}} by: Given $u_k \in X_h, \ p_k \in Q_h$, solve for $u_{k+1}\in X_h,\ p_{k+1}\in Q_h$ satisfying
\begin{align}
b^*(u^{m},{\color{black} u^{m+1}},v) - (p^{m+1},\nabla \cdot v) + \nu(\nabla u^{m+1},\nabla v) & = (f,v) \ \forall v\in X_h, \label{ns4} \\
\epsilon (p^{m+1},q) + (\nabla \cdot u^{m+1},q) & = \epsilon (p^m,q) \ \forall q\in Q_h, \label{ns5} 
\end{align}
where $\epsilon>0$ is a penalty parameter {\color{black}which is }generally taken small.  It is proven by Codina in \cite{C93} that if the penalty parameter 
$\epsilon < \frac{\nu \beta^2}{M^2}$ and {\color{black} small data condition $\kappa<1$ holds} then both $\| p^{m+1} - p \|$ and $\| \nabla( u^{m} - u) \|$ converge linearly to $0$, with rate at most 
\begin{equation}
\mbox{rate}_{IPP} = \left( \frac12 + \frac12 \left(1+\frac{2}{r} \right)^{1/2} \right) \left( \kappa + \epsilon \nu^{-1} \beta^{-2} (M C_1 + r \nu)^2 \right), \label{IPPrate}
\end{equation}
where 
\[
C_1 = \left(\nu^{-1} \| f\|_{-1} + \epsilon^{1/2} \nu^{-1/2} \left(\alpha \nu \beta^{-1} \| \nabla( u - u^0) \| + \| p - p^0 \| + \| p \| \right) \right) / \left(1- \epsilon^{1/2}\nu^{-1/2} \beta^{-1} M \right)
\]
and ${\color{black}r\geq 2}$ (the optimal $r\ge 2$ will depend on the other parameters).  While this is a complex expression,
a rate closer to $\kappa$ {\color{black}--}the convergence rate of the usual Picard iteration \cite{GR86} which is recovered when $\epsilon=0${\color{black}--} is typically observed.  However, even with larger penalty parameter such as $\epsilon=1$, IPP enhanced with Anderson acceleration can be very effective, even for larger data \cite{RVX21}.


We now establish that if  $\gamma=\rho\alpha^{-1}$, $\rho=\nu^{-1}$ and $\alpha=\frac{\epsilon}{\nu}$, then the grad-div stabilized AH method is identical to IPP.   From \eqref{ah2}, since $\nabla \cdot X_h=Q_h$, we observe that $p^{m+1} \equiv p^m - \rho \alpha^{-1} \nabla \cdot u^{m+1} \equiv p^m- \gamma \nabla \cdot u^{m+1}$, by taking $\gamma=\rho\alpha^{-1}$.  Substituting into
\eqref{ah1}, we obtain
\[
\frac{1}{\rho} (\nabla (u^{m+1} - u^m),\nabla v) + \nu (\nabla u^m,\nabla v) + b^*(u^m ,u^{m+1},v) - ( p^{m+1},\nabla \cdot v)  = (f,v), 
\]
and so setting $\rho=\nu^{-1}$ yields
\[
 \nu (\nabla u^{m+1},\nabla v) + b^*(u^m ,u^{m+1},v) - ( p^{m+1},\nabla \cdot v)  = (f,v).
\]
Note that the grad-div stabilized AH momentum equation now {\color{black}exactly }matches {\color{black}with} the IPP momentum equation \eqref{ns4}.  Substituting $\rho=\nu^{-1}$ into \eqref{ah2} we obtain
\[
\nu\alpha(p^{m+1} - p^m,q) + (\nabla \cdot u^{m+1},q)=0.
\]
Finally setting $\alpha=\frac{\epsilon}{\nu}$, we recover \eqref{ns5}, thus {\color{black}establishment of} the grad-div stabilized AH and IPP {\color{black}methods} are equivalent when parameters are chosen so that $\gamma=\rho\alpha^{-1}$, $\rho=\nu^{-1}$ and $\alpha=\frac{\epsilon}{\nu}$.  {\color{black}With this} connection, we have proved the following theorem for grad-div stabilized AH {\color{black}method}.
\begin{theorem} \label{thm1}
Suppose grad-div stabilized AH is computed with parameters  $\rho=\nu^{-1}$, $\alpha=\frac{\epsilon}{\nu}$ and $\gamma=\rho\alpha^{-1} = \epsilon^{-1}$, with user selected penalty parameter $\epsilon < \frac{\nu \beta^2}{M^2}$.  Then, {\color{black}under the small data condition}  $\kappa<1$, grad-div stabilized AH {\color{black}method }converges linearly with {\color{black}a} rate at most rate$_{IPP}$ defined in \eqref{IPPrate}.
\end{theorem}
\begin{remark}
Just as with IPP,  the convergence rate of $\kappa$ will typically be observed in practice {\color{black}for sufficiently small $\epsilon$.}  Moreover, with AA, larger penalty parameters such as $\epsilon=1$ can be used and yield a very efficient and effective iteration even for $\kappa>1$ as shown in \cite{RVX21}.
\end{remark}
\begin{remark}
While we do not prove {\color{black}it in our current manuscript}, we expect {\color{black}that the} parameters `near' those in the theorem will still provide {\color{black}a} linear convergence by continuity.   We {\color{black}believe that it can be proved by using the similar  argument followed in \cite{C93}, however the additional terms will create a quite challenging theory } that is 
beyond the scope of this paper.  Our numerical tests show that grad-div stabilized AH is effective with a rather wide range of parameter choices, especially for smaller $\epsilon$ and with AA. 
\end{remark}

\subsection{Improvement to classical analysis of AH {\color{black}method} via grad-div stabilization}

We now consider the improvements to the classical convergence arguments for AH {\color{black}method}, without assuming certain choices of finite elements or meshes other than $X_h \times Q_h$ {\color{black}which} satisfy the inf-sup condition.  Begin the analysis by adding and subtracting the true solution {\color{black}$(u,p)$} from \eqref{ah1} - \eqref{ah2}, and then subtracting \eqref{Eq:FEM_VP} yields

\begin{multline}\label{AHVGD1}
	\frac{1}{\rho}(\nabla(u^{m+1} -u + u - u^{m}),\nabla v ) +\nu(\nabla(u^{m} - u),\nabla v ) +  b^*(u^m,u^{m+1}, v )\\
	- b^*(u,u, v ) + \gamma (\nab\cdot(u^{m+1} - u), \nab\cdot v )-(\nab\cdot v , p^m - p)=0,
\end{multline}

\color{black}
and
\begin{equation}\label{AHPGD1}
\alpha((p^{m+1} - p + p - p^{m}),q) + \rho(\nab\cdot(u^{m+1} - u + u), q)=0,
\end{equation}
respectively. 

Denote $e^{m+1} = u^{m+1} - u$ and $e_p^{m+1} = p^{m+1} - p$. Taking $ v  = e^{m+1}$ and $q = e_p^{m+1}$ in \eqref{AHVGD1}-\eqref{AHPGD1} respectively gives us
\begin{multline}\label{AHVGD2}
\frac{1}{\rho}(\nabla(e^{m+1} - e^{m}),\nabla e^{m+1}) +\nu(\nabla e^{m},\nabla e^{m+1}) +  b^*(u^m,u^{m+1},e^{m+1})\\
- b^*(u, u ,e^{m+1}) + \gamma (\nab\cdot e^{m+1}, \nab\cdot e^{m+1})-(\nab\cdot e^{m+1}, e_p^m)=0,
\end{multline}
\begin{equation}\label{AHPGD2}
\alpha((e_p^{m+1} - e_p^{m}),e_{\color{black}p}^{m+1}) + \rho(\nab\cdot e^{m+1}, e_p^{m+1})=0.
\end{equation}
\color{black}
Since $b^*$ is skew-symmetric with respect to its last two arguments, 
\begin{align*}
b^*(u^m,u^{m+1},e^{m+1}) - b^*(u,u,e^{m+1}) &= b^*(u^m,u^{m+1},u^{m+1}) - b^*(u^m,u^{m+1},u) - b^*(u,u,u^{m+1}),\\
&= b^*(e^m,u,u^{m+1}),\\
&= b^*(e^m,u,u^{m+1}) - b^*(e^m,u,u),\\
&= b^*(e^m,u,e^{m+1}).
\end{align*} 
Next adding \eqref{AHVGD2}-\eqref{AHPGD2}, using the polarization identity, rearranging and simplifying {\color{black}the} terms, we get 
\begin{multline}\label{AHVGD3}
\frac{1}{2\rho}(\|\nabla(e^{m+1} - e^{m})\|^2 + \|\nabla e^{m+1}\|^2) +\nu\|\nabla e^{m+1}\|^2 + \gamma \|\nab\cdot e^{m+1}\|^2  + \frac{\alpha}{2\rho}(\|e_p^{m+1} - e_p^m \|^2 + \|e_p^{m+1}\|^2)\\
= \frac{1}{2\rho}\|\nabla e^{m}\|^2 + \frac{\alpha}{2\rho}\|e_p^m\|^2 + \nu(\nabla(e^{m+1} - e^{m}),\nabla e^{m+1}) - b^*(e^m,u,e^{m+1}) + (\nabla\cdot e^{m+1}, e_p^{m} - e_p^{m+1}).
\end{multline}
To bound {\color{black}the nonlinear term on the right hand side of last relation,} we use \eqref{skew2} and \eqref{nsebound}:
\[
| - b^*(e^m,u,e^{m+1})  | \le M \| \nabla e^m \| \| \nabla u \| \| \nabla e^{m+1} \| \le M\nu^{-1} \| f \|_{-1}  \| \nabla e^m \|  \| \nabla e^{m+1} \| = \kappa \nu  \| \nabla e^m \|  \| \nabla e^{m+1} \|.
\]
For the {\color{black} third and the last terms} in \eqref{AHVGD3}, we utilize Cauchy-Schwarz and Young's inequalities via
\begin{multline}\label{AHVGD4}
	\nu(\nabla(e^{m+1} - e^{m}),\nabla e^{m+1})  + (\nabla\cdot e^{m+1}, e_p^{m} - e_p^{m+1})\\
	\leq \frac{1}{4\rho}\|\nabla(e^{m+1} - e^{m})\|^2 + \rho\nu^2 \|\nabla e^{m+1}\|^2 
	+ \frac{\alpha}{4\rho}\|e_p^{m} - e_p^{m+1} \|^2 + \frac{\rho}{\alpha}\|\nabla\cdot e^{m+1}\|^2.
\end{multline}
Using {\color{black}Young's and the triangle inequalities} provides
\[
\kappa \nu \|\nabla e^{m}\| \|\nabla e^{m+1}\| \le \kappa\nu(\|\nabla(e^{m+1} - e^{m})\| + \|  \nabla e^{m+1} \|)\|\nabla e^{m+1}\| \le 
 \frac{\kappa\nu}{2}\|\nabla(e^{m+1} - e^{m})\|^2 + \frac{3\kappa\nu}{2}\|\nabla e^{m+1}\|^2,
\]
and now combining the above bounds, we obtain
\begin{multline}\label{AHVGD6}
	\left(\frac{1}{2\rho} - \frac{\kappa\nu}{2}\right)\|\nabla(e^{m+1} - e^{m})\|^2 + \left(\frac{1}{2\rho}  + \nu - \rho\nu^2- \frac{3\kappa\nu}{2}\right)\|\nabla e^{m+1}\|^2\\ + \left(\gamma-\frac{\rho}{\alpha}\right)\|\nabla\cdot e^{m+1}\|^2+ \frac{\alpha}{2\rho}\| e_p^{m+1} \|^2 +\frac{\alpha}{4\rho}\|e_p^{m} - e_p^{m+1} \|^2 
	\leq \frac{1}{2\rho}\|\nabla e^{m}\|^2 + \frac{\alpha}{2\rho}\| e_p^{m} \|^2.
\end{multline}

Provided {\color{black}that} $\rho \le \nu^{-1} \max \left\{ \kappa^{-1},1 \right\}$ and $\gamma\ge \frac{\rho}{\alpha}$ along {\color{black}with} the additional small data assumption $\kappa<\frac23$, the estimate \eqref{AHVGD6} is 
sufficient to provide convergence of the grad-div stabilized AH method.  Comparing to analysis without the grad-div term from \cite{CHS17}, {\color{black} we observe that with grad-div the coefficient of the left hand side term $\| \nabla e^{m+1} \|^2$ is larger and there are less restrictions on the parameters (including no restriction now on $\alpha$).}  The key difference arises from utilizing the
left hand side term $\| \nabla \cdot e^{m+1} \|$, allowing for a larger coefficient of $\| \nabla \cdot e^{m+1} \|$.  Since this is not a proof of contraction, it offers less of a comparison of rates than the previous section did. {\color{black}Note that} if $\nabla \cdot X_h \subset Q_h$ then we can follow the proof of \cite{CHS17} to prove a contraction, however these results would be similar to that of \cite{CHS17} and not nearly as strong as what is proven above in Theorem \ref{thm1}.

\section{Anderson Acceleration applied to the grad-div stabilized AH Method}

In this section, we show that the Anderson acceleration (AA) method can be applied to the grad-div stabilized AH method (Algorithm \ref{Alg.3}) and will improve its linear convergence rate.  We begin this section with a review of AA and recent theoretical results. We will proceed to show how the grad-div stabilized AH method fits into this framework, which in turn allows for invoking the AA theory.  Throughout this section, we assume that the data is sufficiently small and parameters are chosen so that grad-div stabilized AH provides a contractive iteration; we specify this assumption below precisely, after we give some notation.

\subsection{Anderson acceleration}


In this subsection, we provide AA procedure and its convergence properties. Consider a fixed-point operator $g:Y\rightarrow Y$ where Y is a Hilbert space equipped with induced norm $\|\cdot\|_Y$, and  denote $w_{j} = g(x_{j-1}) - x_{j-1}$ as the nonlinear residual, also sometimes is called the update step. Then, the AA algorithm with depth $m$ (if $m=0$, it returns to usual Picard iteration) applied to the fixed-point problem $ g(x) = x $, reads as follows.

\begin{algorithm} \label{alg:anderson}
	(Anderson acceleration with depth $m$ and damping factors $\beta_k$)\\ 
	Step 0: Choose $x_0\in Y.$\\
	Step 1: Find $w_1\in Y $ such that $w_1 = g(x_0) - x_0$.  
	Set $x_1 = x_0 + w_1$. \\
	Step $k$: For $k=2,3,\ldots$ Set $m_k = \min\{ k-1, m\}.$\\
	\indent [a.] Find $w_{k} = g(x_{k-1})-x_{k-1}$. \\
	\indent [b.] Solve the minimization problem for the Anderson coefficients $\{ \alpha_{j}^{k}\}_{k-m_k}^{k-1}$
	\begin{align}\label{eqn:opt-v0}
		\{ \alpha_{j}^{k}\}_{k-m_k}^{k-1}=\textstyle \text{argmin} 
		\left\| \left(1- \sum\limits_{j=k-m_k}^{k-1} \alpha_j^{k} \right) w_k + \sum\limits_{j = k-m_k}^{k-1}  \alpha_j^{k}  w_{k-j} \right\|_Y.
	\end{align}
	\indent [c.] For damping factor $0 < \beta_k \le 1$, set
	\begin{align}\label{eqn:update-v0}
		\textstyle
		x_{k} 
		= (1-\sum\limits_{j = k-m_k}^{k-1}\alpha_j^k) x_{k-1} + \sum_{j= k-m_k}^{k-1} \alpha_j^{k} x_{j-1}
		+ \beta_k \left(  (1- \sum\limits_{j= k-m_k}^{k-1} \alpha_j^{k}) w_k + \sum\limits_{j=k-m_k}^{k-1}\alpha_j^k w_{k-j}\right).
	\end{align}
\end{algorithm}
To understand how AA improves convergence, we define the optimization gain factor $\theta_k$ by
 \begin{align*}
\theta_{k}=\frac{\nr{ (1- \sum\limits_{j= k-m_k}^{k-1} \alpha_j^{k}) w_k + \sum\limits_{j=k-m_k}^{k-1}\alpha_j^k w_{k-j}}_Y}{\|w_k\|_Y},
 \end{align*}
which characterizes the improvement in fixed-point convergence rate as proposed in \cite{PR21, PRX19}. 

The following assumptions from \cite{PR21} provide sufficient conditions on the fixed point operator $g$ for the convergence and acceleration results.
\begin{assumption}\label{assume:g}
	Assume $g\in C^1(Y)$ has a fixed point $x^\ast$ in $Y$, and there are positive constants $C_0$ and $C_1$ with
	\begin{enumerate}
		\item $\|g'(x)\|_{Y}\le C_0$ for all $x\in Y$, and 
		\item $\|g'(x) - g'(y)\|_{Y} \le C_1 \|x-y\|_{Y}$
		for all $x,y \in Y$.
	\end{enumerate}
\end{assumption}
\begin{assumption}\label{assume:fg} 
	Assume there is a constant $\sigma> 0$ for which the differences between consecutive
	residuals and iterates satisfy
	\begin{align} \label{eqn:assumefg}
		\| w_{{k}+1} - w_{k}\|_Y  \ge \sigma \| x_{k} - x_{{k}-1} \|_Y, \quad {k} \ge 1.
	\end{align}
\end{assumption}
Assumption \ref{assume:g} will be verified for Picard fixed-point operator grad-div stabilized AH method in next sections. Also, Assumption \ref{assume:fg} can be verified easily for this method which is contractive under small data and particular parameter choices. Under Assumptions \ref{assume:g} and \ref{assume:fg}, the following result from \cite{PR21}, generates a bound on the residual $\|w_{k+1}\|$ in terms of the previous residual $\|w_k\|$.

\begin{theorem}[Pollock et al., 2021]  \label{thm:genm}
	Let Assumptions \ref{assume:g} and \ref{assume:fg} hold,
	and suppose the direction sines between each column $j$ of matrix
	\begin{align*}
	F_j&:=\left( (w_{j}-w_{j-1})(w_{j-1}-w_{j-2})\ \dotsc\ (w_{j-m_j+1}-w_{j-m_j})  \right)
\end{align*}
	and the subspace spanned by the preceeding columns satisfy
	$|\sin(f_{j,i},\text{span }\{f_{j,1}, \ldots, f_{j,i-1}\})| \ge c_s >0$,
	for $j = k-m_k, \ldots, k-1$.
	Then the residual $w_{k+1} = g(x_k)-x_k$ from Algorithm \ref{alg:anderson} 
	(depth $m$) satisfies the following bound.
	\begin{align}\label{eqn:genm}
		\nr{w_{k+1}} & \le \nr{w_k} \Bigg(
		\theta_k ((1-\beta_{k}) + C_0 \beta_{k})
		+ \frac{C C_1\sqrt{1-\theta_k^2}}{2}\bigg(
		\nr{w_{k}}h(\theta_{k})
		\nonumber \\ &
		+ 2  \sum_{n = k-{m_{k}}+1}^{k-1} (k-n)\nr{w_n}h(\theta_n) 
		+ m_{k}\nr{w_{k-m_{k}}}h(\theta_{k-m_{k}})
		\bigg) \Bigg),
	\end{align}
	where  each $h(\theta_j) \le C \sqrt{1 - \theta_j^2} + \beta_{j}\theta_j$,
	and $C$ depends on $c_s$ and the implied upper bound on the direction cosines. 
\end{theorem}

In \eqref{eqn:genm}, the optimization gain $\theta_k$ is scaling the first-order term, which is residual in the standard fixed-point iteration. On the other hand, the higher-order terms are scaled by a factor of $\sqrt{1-\theta_k^2}$, which implies if the optimization works, the relative weight of the higher-order terms increase, otherwise the relative weight of the first-order term increase in \eqref{eqn:genm}.

%
%
%
%
%
%
%
%
%


\subsection{Grad-div stabilized AH method as a fixed point iteration}


In this subsection, we define the fixed-point operator $G$ which is associated with grad-div stabilized AH iteration.  Note that in this section, $u,p$ are generic functions and are not the steady NSE solution as in the previous section. 
	
\begin{definition}
Define mapping $G:(X_h, Q_h)\rightarrow (X_h, Q_h) $, $G(u,p)=(G_1(u,p),G_2(u,p))$ such that for any $(v,q)\in(X_h, Q_h)$
\begin{align}
	\rho^{-1}(\nabla(G_1(u,p)-u),\nabla v) + \nu(\nabla u, \nabla v) + b(u;G_1(u,p),v) \nonumber\\+ \gamma(\nabla\cdot G_1(u,p), \nabla\cdot v)
	- (p, \nabla\cdot v)&= (f,v),\label{AHwithG}\\
	\alpha(G_2(u,p)-p,q) + \rho (\nabla\cdot G_1(u,p),q)&=0.\label{AHwithG1}    
\end{align} 
\end{definition}
Now, we show that $G$ is well-defined and bounded with respect to the norm $\|( v ,q)\|_H\eqqcolon \sqrt{\|\nabla  v \|^2 + \alpha \|q\|^2}$ on $ X_h \times Q_h$.
\begin{lemma}\label{lemma:Gwelldefined}The operator $ G $ is well-defined. Moreover, if $G(u,p)=(G_1(u,p),G_2(u,p))$ is the solution of \eqref{AHwithG}-\eqref{AHwithG1}, then the following inequality holds
	\begin{align}\label{Gwelldefined}
		\|G(u,p)\|_H
		\leq 
		\sqrt{\frac{(1+4\rho^2\nu^2 )}{N}}\|\nabla u\|
		+\sqrt{ \frac{\alpha}{N}}\|p\|
		+\sqrt{  \frac{4\rho^2}{N}} \|f\|_{-1},
	\end{align}
	where $N\coloneqq \min\{  \frac{1}{2}-\alpha^{-1}\rho^2,1 \}$.
\end{lemma}
\begin{proof}
	Assume that a solution exists. Then, choosing $ v =G_1(u,p)$ and $q=G_2(u,p)$ eliminates the nonlinear term and yields
	\begin{align*}
		\frac{1}{2} \left( \|\nabla G_1(u,p)\|^2 - \|\nabla u\|^2 + \|\nabla(G_1(u,p)-u)\|^2\right)
		+  \rho\gamma\|\nabla\cdot G_1(u,p)\|^2-\rho(p,\nabla\cdot G_1(u,p)) \\=   \rho(f,G_1(u,p))- \rho\nu(\nabla u, \nabla G_1(u,p)),\\
		\frac{\alpha}{2}\left( \|G_2(u,p)\|^2 - \|p\|^2 + \|G_2(u,p)-p\|^2 \right) + \rho (\nabla\cdot G_1(u,p),G_2(u,p))=0.
	\end{align*} 
thanks to the polarization identity. Then, combining the above equations gives that
	\begin{align*}
		&\frac{1}{2} \left( \|\nabla G_1(u,p)\|^2 + \|\nabla(G_1(u,p)-u)\|^2\right) 
		+\frac{\alpha}{2}\left( \|G_2(u,p)\|^2 + \|G_2(u,p)-p\|^2 \right)
		+\rho\gamma\|\nabla\cdot G_1(u,p)\|^2
		\\
		&= 
		\frac{1}{2} \|\nabla u\|^2
		+ \frac{\alpha}{2}\|p\|^2
		+ \rho(f,G_1(u,p))
		- \rho\nu(\nabla u, \nabla G_1(u,p))
		-\rho(G_2(u,p)-p,\nabla\cdot G_1(u,p)) .
	\end{align*}
	Multiplying the both sides of the last relation by $2$, dropping positive terms $\|\nabla(G_1(u,p)-u)\|^2$, $\rho\gamma\|\nabla\cdot G_1(u,p)\|^2$, and $\|G_2(u,p)-p\|^2$ on the left hand side, and using H\"older's and Young's inequalities produce
	\begin{align*}
		\left( \frac{1}{2}-\alpha^{-1}\rho^2 \right)\|\nabla G_1(u,p)\|^2 
		+\alpha \|G_2(u,p)\|^2 
		\leq 
		(1+4\rho^2\nu^2 )
		\|\nabla u\|^2
		+ \alpha\|p\|^2
		+ \rho^2 4 \|f\|^2_{-1}.
	\end{align*}
	Letting $N\coloneqq \min\{  \frac{1}{2}-\alpha^{-1}\rho^2,1 \}$, and dividing both sides by $N$, we get
	\begin{align*}
		\|G(u,p)\|^2
		\leq 
		\frac{(1+4\rho^2\nu^2 )}{N}\|\nabla u\|^2
		+ \frac{\alpha}{N}\|p\|^2
		+ \frac{4\rho^2}{N} \|f\|^2_{-1} .
	\end{align*}
	Then, taking the square root of both sides reduces it to \eqref{Gwelldefined}. Since $G$ is linear and finite dimensional, showing that the solution $G(u,p)$ is bounded continuously by the data implies solution uniqueness and thus existence as well.

\end{proof}

We now rewrite the grad-div stabilized AH method in terms of a mapping $G:(X_h,Q_h)\rightarrow(X_h,Q_h)$ that satisfies for $m\geq 0$
\begin{equation*}
	G(u^m,p^{m})=(G_1(u^m,p^{m}),G_2(u^m,p^{m})) := (u^{m+1},p^{m+1}),
\end{equation*}
where $(u^m,p^{m})$ is the $m^{th}$ iteration of the A-H method described in Algorithm \ref{Alg.3}.

\subsection{Applying AA to the grad-div stabilized AH iteration}

In this subsection, we show the sufficient smoothness properties of the associated fixed point operator $G$ for the grad-div stabilized AH iteration to apply AA theory. Now, we show Lipschitz continuity of $ G $.
\begin{lemma}\label{lemma:GLipschitz} For any $(u,p), (w,z)\in  (X_h, Q_h)$, we have
\begin{align}\label{GLipschitz}
	\|G(u,p)-G(w,z)\|_H \leq C_L \|(u,p)-(w,z)\|_H                                                                 
\end{align}
where
$C_L= \max \{ \sqrt{\frac{2}{K}}, \left(\frac{1-\rho\nu}{\sqrt{K}}+ \frac{\rho }{\sqrt{KN}}\left((1+2\rho\nu)\|\nabla u\|
+ \sqrt{\alpha}\|p\|
+ 2\rho\|f\|_{-1} \right)\right)  \}$.
\end{lemma}

\begin{remark}
Note that we have already discussed that in case of sufficiently small data and particularly chosen parameters, Algorithm \ref{Alg.3} is contractive.  In this section we assume to be in the contractive setting, and thus we have that the Lipschitz constant $C_L$ in Lemma \ref{lemma:GLipschitz} is less than $1$.
\end{remark}

\begin{proof}
Subtracting \eqref{AHwithG} with $(w,z)$ from \eqref{AHwithG} with $ (u,p)$ gives
\begin{align*}
	(\nabla(G_1(u,p)-G_1(w,z)),\nabla  v )-(\nabla(u-w),\nabla  v ) + \rho\nu(\nabla(u-w), \nabla  v ) + \rho b(u;&G_1(u,p)-G_1(w,z), v )\\+ \rho b(u-w;G_1(w,z), v )+\rho\gamma(\nabla\cdot(G_1(u,p)- G_1(w,z)), \nabla\cdot  v )&=\rho(p-z,\nabla\cdot  v ) ,\\
	\alpha(G_2(u,p)-G_2(w,z),q) -\alpha(p-z,q)+ \rho (\nabla\cdot  (G_1(u,p)-G_1(w,z)),q)&=0  .
\end{align*}

Then, setting $v=G_1(u,p)-G_1(w,z)$ and $q=G_2(u,p)-G_2(w,z)$ which eliminates the first nonlinear term on the left hand side of the first equation, and combining these equations provide
\begin{align*}
	&\|\nabla(G_1(u,p)-G_1(w,z))\|^2+\alpha\|G_2(u,p)-G_2(w,z)\|^2+\rho\gamma\|\nabla\cdot(G_1(u,p)- G_1(w,z))\|^2\\& =\rho(p-z,\nabla\cdot (G_1(u,p)-G_1(w,z)))+(1-\rho\nu)(\nabla(u-w),\nabla (G_1(u,p)-G_1(w,z))) \\&-\rho b(u-w;G_1(w,z),G_1(u,p)-G_1(w,z))+\alpha(p-z,G_2(u,p)-G_2(w,z))\\&-\rho (\nabla\cdot (G_1(u,p)-G_1(w,z)),G_2(u,p)-G_2(w,z)) .                                                            
\end{align*}
By dropping positive term $\rho\gamma\|\nabla\cdot(G_1(u,p)- G_1(w,z))\|^2$ on the left hand side and using Cauchy-Schwarz and Young's inequalities, Lemma \ref{lemma:Gwelldefined} and {\color{black}\eqref{skew2}} , we get
\begin{align*}
	\left(\frac{1}{2}-\frac{3\rho^2\alpha^{-1}}{4}\right)\|\nabla(G_1(u,p)-G_1(w,z))\|^2+\frac{\alpha}{4}\|G_2(u,p)-G_2(w,z)\|^2
	 \leq
	2\alpha\|p-z\|^2\\
	+\left((1-\rho\nu)^2+ \rho^2 \left( \frac{(1+4\rho^2\nu^2 )}{N}\|\nabla u\|^2
	+ \frac{\alpha}{N}\|p\|^2
	+ \frac{4\rho^2}{N} \|f\|^2_{-1} \right)\right) \|\nabla(u-w)\|^2.                                                                    
\end{align*}
Defining $K=\min \{\left(\frac{1}{2}-\frac{3\rho^2\alpha^{-1}}{4}\right), \frac{1}{4}   \}$ and then dividing both sides by $K$ yields \eqref{GLipschitz}.
\end{proof}

Next, we define an operator $G'$ and show it is indeed the Fr\'echet derivative of the operator of $G.$
\begin{definition}
Given $(u,p)\in  X _h\times Q_h$, define an operator $G'(u,p;\cdot,\cdot): X_h\times Q_h\rightarrow  X_h\times Q_h$ by
\begin{align*}
G'(u,p;w,s)\eqqcolon (G'_1(u,p;w,s),G'_2(u,p;w,s))
\end{align*}
satisfying for all $(w,s)\in X _h\times Q_h$.
\begin{equation}\label{Gderivative}
	\begin{aligned}
		(\nabla(G'_1(u,p;w,s)-w),\nabla  v ) + \rho\nu(\nabla w, \nabla  v ) +& \rho b(w;G_1(u,p), v )+ \rho b(u;G'_1(u,p;w,s), v )\\ \rho\gamma(\nabla\cdot G'_1(u,p;w,s), \nabla\cdot  v )-\rho(s,\nabla\cdot  v )&= 0,\\
		\alpha(G'_2(u,p;w,s)-s,q) + \rho (\nabla\cdot G'_1(u,p;w,s),q)&=0.    
	\end{aligned}                                                                    
\end{equation} 
\end{definition}
\begin{lemma}\label{lemma:Gderivative}
The operator $G'$ is well-defined for all $(u,p), (w,s)\in  X_h\times Q_h$ such that
\begin{equation}\label{Gderivativebound}
	\begin{aligned}
		\| G'(u,p;w,s)\|_H\leq {\color{black}C_L}\|(w,s)\|_H
	\end{aligned}                                                                    
\end{equation} 
\end{lemma}
\begin{proof}
	Adding equations in \eqref{Gderivative}  and setting $ v =G'_1(u,p;w,s)$ and $q=G'_2(u,p;w,s)$ produces
\begin{equation}
	\begin{aligned}
		&\|\nabla G'_1(u,p;w,s)\|^2 
		+ \rho\gamma\|\nabla\cdot G'_1(u,p;w,s)\|^2
		+\alpha\|G'_2(u,p;w,s)\|^2 
		\\&=
		(1-\rho\nu)(\nabla w,\nabla G'_1(u,p;w,s))
		+
		\alpha(s,G'_2(u,p;w,s))
		- \rho b(w;G_1(u,p),G'_1(u,p;w,s)) 
				\\&
		+\rho(s,\nabla\cdot G'_1(u,p;w,s))
		-\rho (\nabla\cdot G'_1(u,p;w,s),G'_2(u,p;w,s)).
	\end{aligned}                                                                    
\end{equation} 
{\color{black}
Then, by dropping positive term $\rho\gamma\|\nabla\cdot G'_1(u,p;w,s)\|^2$ on the left hand side, applying Cauchy-Schwarz and Young's inequalities produces
%
%
\begin{align*}
&\|\nabla G'_1(u,p;h,s)\|^2 
+\alpha\|G'_2(u,p;h,s)\|^2 
\\&\leq
(1-\rho\nu)^2\|\nabla h\|^2+\frac{1}{4} \|\nabla G'_1(u,p;h,s)\|^2
+
\alpha\|s\|^2 +\frac{\alpha}{4} \|G'_2(u,p;h,s)\|^2
\\&
+ \rho^2 M^2\left( \frac{(1+4\rho^2\nu^2 )}{N}\|\nabla u\|^2
+ \frac{\alpha}{N}\|p\|^2
+ \frac{4\rho^2}{N} \|f\|^2_{-1} \right) \|\nabla h\|^2  +\frac{1}{4} \|\nabla G'_1(u,p;h,s)\|^2
\\&
+\alpha\|s\|^2 +\frac{\rho^2\alpha^{-1}}{4} \|\nabla G'_1(u,p;h,s)\|^2
+\frac{\rho^2\alpha^{-1}}{2} \|\nabla G'_1(u,p;h,s)|^2 +\frac{\alpha}{2} \|G'_2(u,p;h,s)\|^2
\end{align*}
thanks to the Lemma \ref{lemma:Gwelldefined} and \eqref{skew2}. After rearranging the terms, we get
}
	\begin{align*}
	&\left(\frac{1}{2}-\frac{3\rho^2\alpha^{-1}}{4}\right)\|\nabla G'_1(u,p;h,s)\|^2 
	+\frac{\alpha}{4} \|G'_2(u,p;h,s)\|^2 
	\\&\leq
	\left( (1-\rho\nu)^2+\rho^2 M^2\left( \frac{(1+4\rho^2\nu^2 )}{N}\|\nabla u\|^2
	+ \frac{\alpha}{N}\|p\|^2
	+ \frac{4\rho^2}{N} \|f\|^2_{-1} \right) \right)\|\nabla h\|^2+
	2\alpha\|s\|^2 
\end{align*}
{\color{black}  
Taking the square root of both sides and considering constant $K$ which is defined in proof of  Lemma \ref{lemma:GLipschitz} finishes the proof.}


\end{proof}

Now, we prove that $G'$ is Fr\'echet derivative operator of $G.$
\begin{lemma}\label{lemma:Frechetderiv}
For any $(u,p)\in X_h\times Q_h$
	\begin{align}\label{Frechetderiv}
	\|\ G(u+w,p+s) - G(u,p) - G'(u,p; w,s)\|_H \leq \frac{\rho M C_L}{\sqrt{1-\rho\alpha^{-1}}} \|(w,s)\|_H^2.
\end{align} 
\end{lemma}
\begin{proof}
Denote $ \eta_1 = G_1(u+w,p+s) - G_1(u,p) - G_1'(u,p; w,s), \eta_2 = G_2(u+w,p+s) - G_2(u,p) - G_2'(u,p;w,s).$
Subtracting the sum of \eqref{Gderivative} and \eqref{AHwithG} from the equation \eqref{AHwithG} with $(u+w, p+s)$ yields


	\begin{align*}
		(\nabla \eta_1,\nabla  v ) +  \rho\gamma(\nabla\cdot \eta_1, \nabla\cdot  v )+\rho b(u;\eta_1, v )+\rho b(w;G_1(u+w,p+s)-G_1(u,p), v )&= 0,\\
		\alpha(\eta_2,q) + \rho (\nabla\cdot \eta_1,q)&=0.    
	\end{align*} 
Choosing $ v =\eta_1$ and $q=\eta_2$ vanishes first nonlinear term on the left hand side, and combining these equations gives
	\begin{equation}
		\|\nabla \eta_1\|^2 + \alpha\|\eta_2\|^2 + \rho\gamma\|\nabla\cdot \eta_1\|^2=- \rho (\nabla\cdot \eta_1, \eta_2)-\rho b(w;G_1(u+w,p+s)-G_1(u,p),\eta_1).
	\end{equation}                       
	Applying Cauchy-Schwarz on the right hand side and dropping $\rho\gamma\|\nabla\cdot \eta_1\|^2$ on the left hand side yields
\begin{align*}
\|\nabla \eta_1\|^2 + \alpha\|\eta_2\|^2 
\leq
\rho \|\nabla\cdot \eta_1\| \| \eta_2\|
+ M \rho \|\nabla w\| \|\nabla G_1(u+w,p+s)-G_1(u,p)\| \|\nabla \eta_1\|.
\end{align*} 
	Now with Young's inequality and Lemma \ref{lemma:GLipschitz}, we obtain using $\|\cdot\nabla\|\leq\|\nabla\|$ and $\|\nabla w\|^2 \leq \|\nabla w\|^2 + \alpha \|s\|^2 = \|(w,s)\|_H^2 $ that
\begin{align*}
\|\nabla \eta_1\|^2 + \alpha\|\eta_2\|^2 
\leq
\frac{\rho^2\alpha^{-1}}{2} \|\nabla \eta_1\|^2 + \frac{\alpha}{2}\| \eta_2\|^2
+\frac{ \rho^2}{2} M^2C_L^2 \|(w,s)\|^4_H + \frac{1}{2}\|\nabla \eta_1\|^2.
\end{align*} 
Then, rearrange and obtain the following

\begin{align*}
\left(\frac{1}{2}-\frac{\rho^2\alpha^{-1}}{2} \right)\|\nabla \eta_1\|^2 + \frac{\alpha}{2}\|\eta_2\|^2 &\leq \frac{\rho^2}{2} M^2 C_L^2\|\nabla w\|^2 \|\nabla (G_1(u+w,p+s)-G_1(u,p))\|^2 \\&\leq \frac{\rho^2}{2} M^2 C_L^2 \|(w,s)\|_H^4,
\end{align*}                                                                    
where we  used \eqref{skew2}. Then, applying the definition of $\eta_1$ and $\eta_2$, dividing both sides by $\left(\frac{1}{2}-\frac{\rho^2\alpha^{-1}}{2} \right)$ and taking square roots give that $G'$ is indeed the Fr\'echet derivative of $G$ which satisfies \eqref{Frechetderiv}.
\end{proof}

We now proceed to show that $G'$ is Lipschitz continuous over $ X_h\times Q_h$.

\begin{lemma}\label{lemma:GLipschitzcont}
G is Lipschitz continuously differentiable on $ X_h\times Q_h$, such that for all $u,w,\theta\in  X _h$ and $p,s,\xi\in Q_h$,
   \begin{equation}\label{GLipschitzcont}
	\begin{aligned}
		\|G'(u+w,p+s;\theta,\xi)-G'(u,p;\theta,\xi)\|_H
		\leq {\color{black}
		\left( \frac{4\rho^2 M^2 C_L^2}{1-\rho^2\alpha^{-1}}\right)^{1/2} }\|(\theta,\xi)\|_H \|(w,s)\|_H.
	\end{aligned}                                                                    
	\end{equation}
\end{lemma}

\begin{proof}
 Subtracting \eqref{Gderivative} with $G'(u,p;\theta,\xi)$ from \eqref{Gderivative} with $G'(u+w, p+s; \theta,\xi)$ and denoting \\$e_1\eqqcolon G'_1(u+w,p+s;\theta,\xi)-G'_1(u,p;\theta,\xi)$ and $e_2\eqqcolon G'_2(u+w,p+s;\theta,\xi)-G'_2(u,p;\theta,\xi)$ yield
 	\begin{align*}
 		(\nabla e_1,\nabla  v ) + \rho b(\theta;G_1(u+w,p+s)-G_1(u,p), v )+ \rho b(u;e_1, v )\\ +\rho b(w;G'_1(u+w,p+s;\theta,\xi), v )+\rho\gamma(\nabla\cdot e_1, \nabla\cdot  v )&= 0,\\
 		\alpha(e_2,q) + \rho (\nabla\cdot e_1,q)&=0.    
 	\end{align*} 

 Setting $ v =e_1$ and $q=e_2$ vanishes the third term on the left hand side of the first equality and adding these equations provide
 	\begin{align*}
 		&\|\nabla e_1\|^2+ \alpha\|e_2\|^2 +\rho\gamma\|\nabla\cdot e_1\|^2
 		\\&=-\rho (\nabla\cdot e_1,e_2)-\rho b(\theta;G_1(u+w,p+s)-G_1(u,p),e_1) -\rho b(w;G'_1(u+w,p+s;\theta,\xi),e_1).
	\end{align*}  

Dropping the positive term $\rho\gamma\|\nabla\cdot e_1\|^2$ on the left hand side, applying Cauchy-Schwarz and Young's inequalities and \eqref{skew2}, using Lemma \ref{lemma:GLipschitz} and \ref{lemma:Gderivative}, we get
 	\begin{align*}
 	\left(\frac{1}{2}-\frac{\rho^2\alpha^{-1}}{2}\right)\|\nabla e_1\|^2+ \frac{\alpha}{2}\|e_2\|^2
 		& \leq
 		\rho^2 M^2 C_L^2 \|\nabla\theta\|^2 \|(w,s)\|^2_H
 		+ \rho^2 M^2 {\color{black}C_L^2 }\|\nabla w\|^2 \|(\theta,\xi)\|^2
 		\\& \leq
 		{\color{black}2\rho^2 M^2 C_L^2} \|(\theta,\xi)\|_H^2 \|(w,s)\|^2_H.
 	\end{align*}                                                                    
Dividing both sides by $\left(\frac{1}{2}-\frac{\rho^2\alpha^{-1}}{2}\right)$ gives
	\begin{align*}
		\|(e_1,e_2)\|_H^2=\|\nabla e_1\|^2+ \alpha\|e_2\|^2
		\leq {\color{black}
		\frac{4\rho^2 M^2 C_L^2}{1-\rho^2\alpha^{-1}} }\|(\theta,\xi)\|_H^2 \|(w,s)\|^2_H.
	\end{align*}                                                                    
 Then taking the square roots of both sides gives that $G'$ is Lipschitz continuous, and \eqref{GLipschitzcont} holds.
\end{proof}

	\subsection{Convergence of the Anderson Accelerated AH algorithm for steady NSE}
In previous subsection, we proved that the solution operator $G$ associated with grad-div stabilized AH iteration \eqref{AHwithG}-\eqref{AHwithG1} satisfies Assumption \ref{assume:g} which is the one of sufficient conditions to apply the one-step residual bound of \cite{PR21}. Also, Assumption \ref{assume:fg} is satisfied since $G$ is contractive under small data condition and certain parameter choices.

Under these assumptions and with Lemmas \ref{lemma:GLipschitz}, \ref{lemma:Frechetderiv}, \ref{lemma:GLipschitzcont} and Theorem \ref{thm:genm}, we have established the convergence of \eqref{AHwithG}-\eqref{AHwithG1} where $G$ is the solution operator associated with grad-div stabilized AH iteration.
	\begin{theorem}
	\label{thm:aa}
	For any step $k>m$ with $\alpha_{m}^k \neq 0$, the following bound holds for the grad-div stabilized AH iteration \eqref{AHwithG}-\eqref{AHwithG1}
	\begin{align*}
		\| (w_{k+1}, z_{k+1})\|_H \le& \theta_k(1-\beta_k + \beta_k C_L) \| (w_k,z_k)\|_H \\
		&+C\sqrt{1-\theta_{k}^2}\|(w_k,z_k)\|_{H} \sum\limits_{j=1}^{m} \|(w_{k-j+1},z_{k-j+1})\|_H,
	\end{align*}
	for the residual $(w_k,z_k)$, where $\theta_k$ is the gain from the optimization problem,  $C_L$ is the Lipschitz constant of $G$ defined in Lemma \ref{lemma:GLipschitz}, and $C$ depending on $\theta_k, \beta_k, C_L$.
\end{theorem}
This theorem tells us that \eqref{AHwithG}-\eqref{AHwithG1}, with a good initial guess, converges linearly with rate $\theta_k(1-\beta_k + \beta_kC_L) <1$, which improves on Algorithm \ref{Alg.3} due to the scaling $\theta_k$ and the damping factor $\beta_k$. In the case $G$ is contractive, i.e. $C_L<1$, then the optimal choice for relaxation is $\beta_k=1$.

\section{Numerical Experiments}\label{numerical}

In this section, we perform several numerical tests to illustrate the theory above and to show how the grad-div stabilized, Anderson accelerated AH algorithm can be an effective and efficient solver
for the steady NSE.  The stopping criteria for all of our tests is $\|u_k - u_{k-1}\|\leq10^{-6}$.

\begin{figure}[h!]
	\centering
	\begin{tabular}{lc}
		&$\gamma=0$\\
		\rot{\quad\quad \quad   Taylor-Hood}
	&\includegraphics[width = .3\textwidth, height=.24\textwidth]{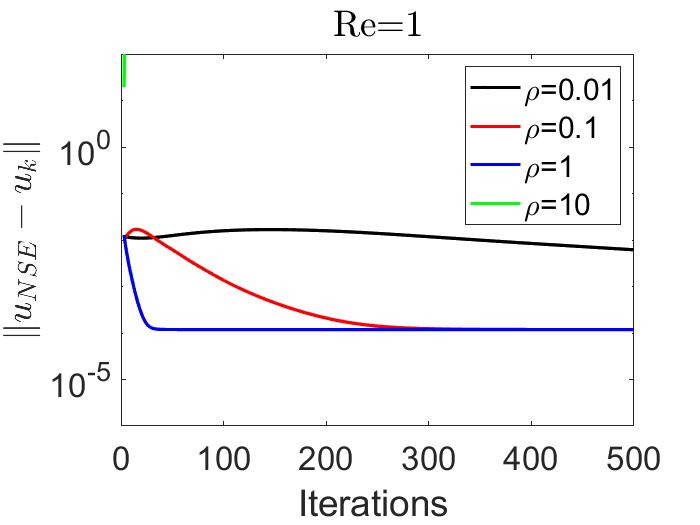}
\includegraphics[width = .3\textwidth, height=.24\textwidth]{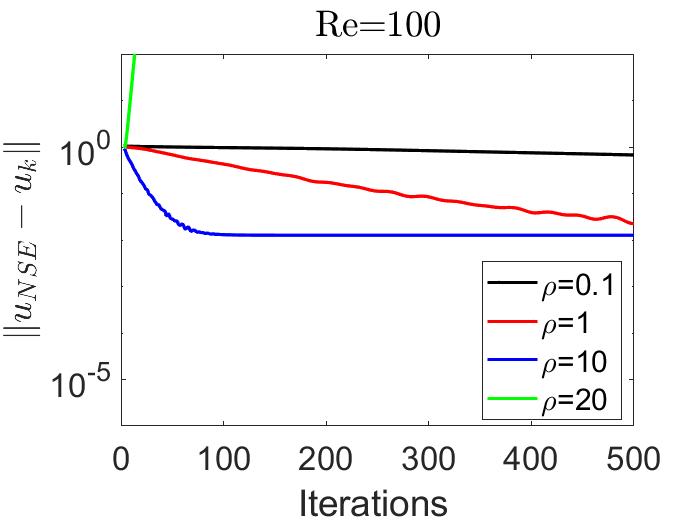}
\includegraphics[width = .3\textwidth, height=.24\textwidth]{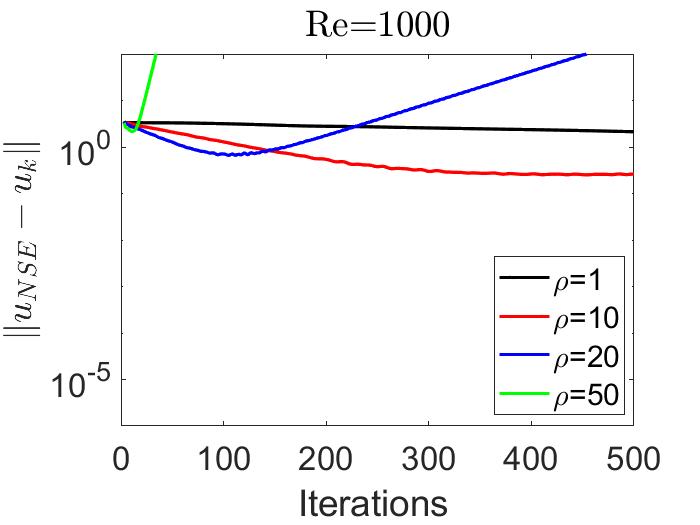}\\
				\rot{\quad\quad \quad   Scott-Vogelius}
		&\includegraphics[width = .3\textwidth, height=.24\textwidth]{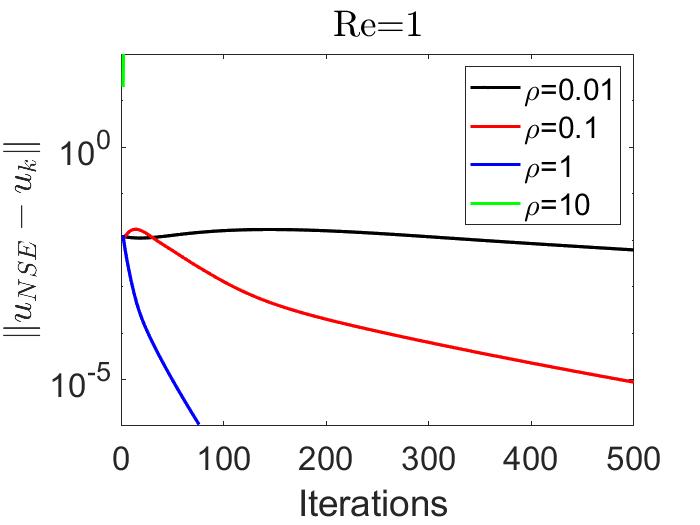}
		\includegraphics[width = .3\textwidth, height=.24\textwidth]{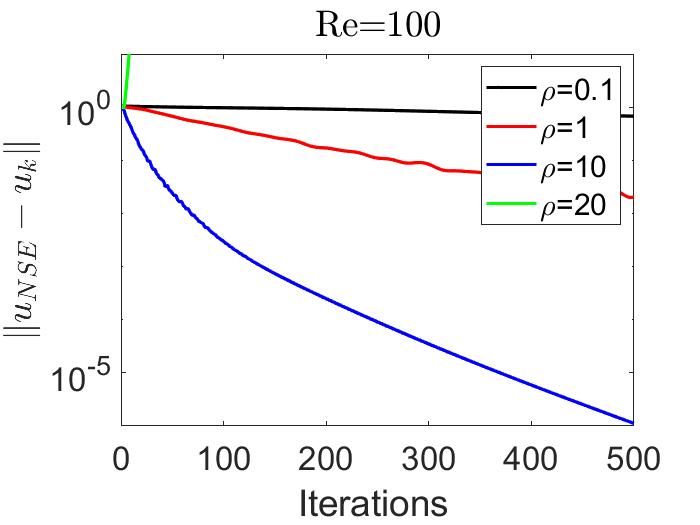}
		\includegraphics[width = .3\textwidth, height=.24\textwidth]{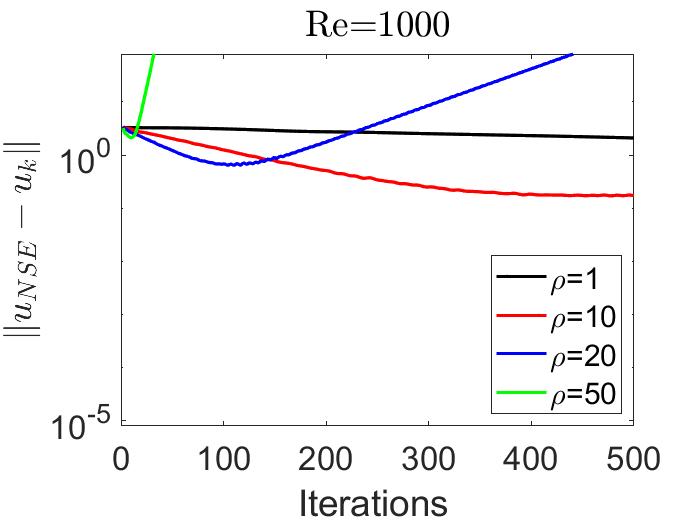}
		\\
	\end{tabular}
\begin{tabular}{lc}
	&$\gamma=1$\\
	\rot{\quad\quad \quad   Taylor-Hood}
	&\includegraphics[width = .3\textwidth, height=.24\textwidth]{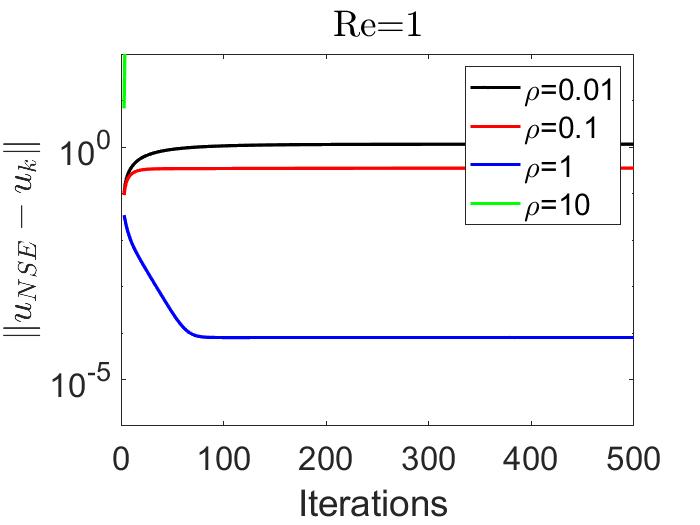}
	\includegraphics[width = .3\textwidth, height=.24\textwidth]{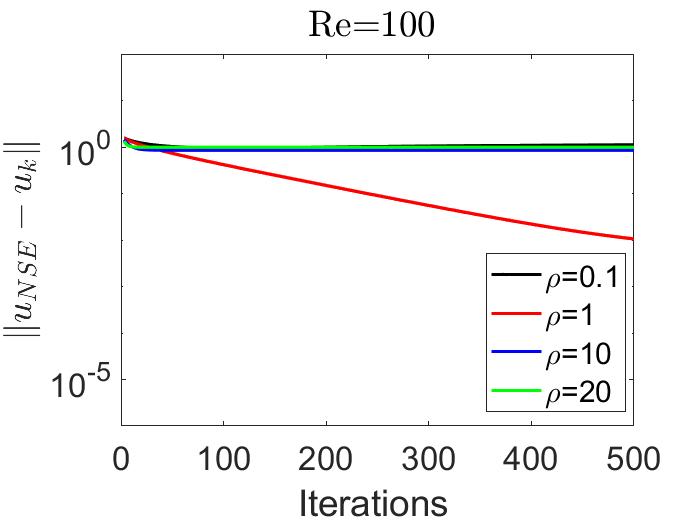}
	\includegraphics[width = .3\textwidth, height=.24\textwidth]{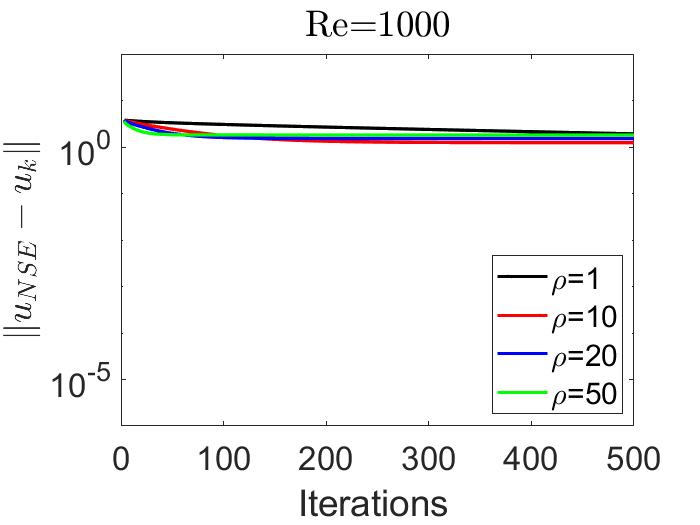}\\
	\rot{\quad\quad \quad   Scott-Vogelius}
	&\includegraphics[width = .3\textwidth, height=.24\textwidth]{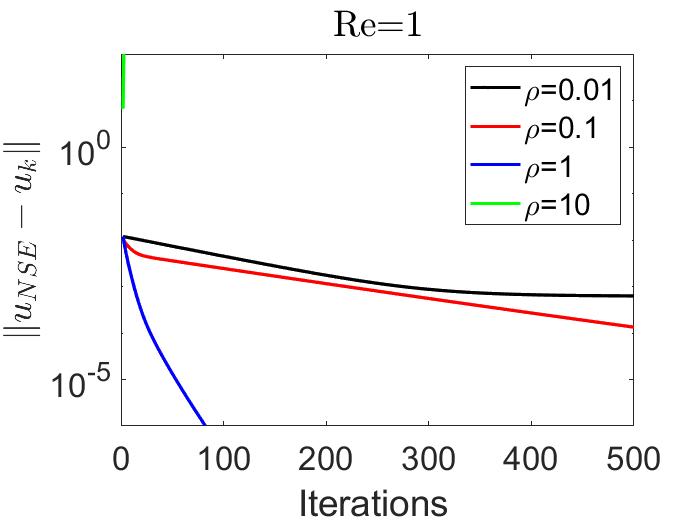}
	\includegraphics[width = .3\textwidth, height=.24\textwidth]{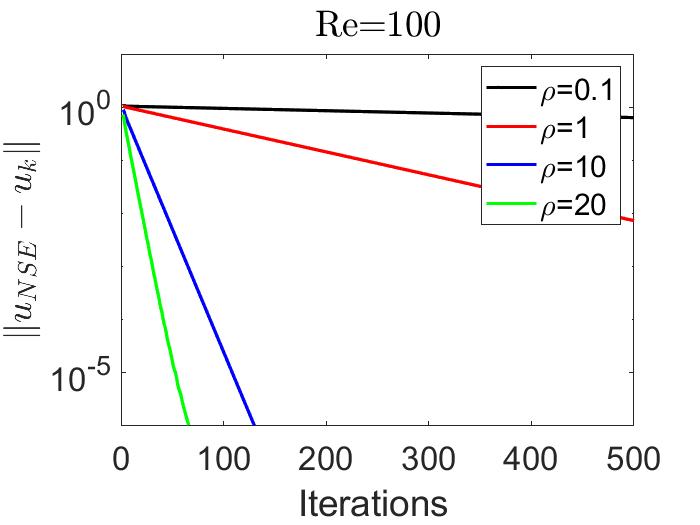}
	\includegraphics[width = .3\textwidth, height=.24\textwidth]{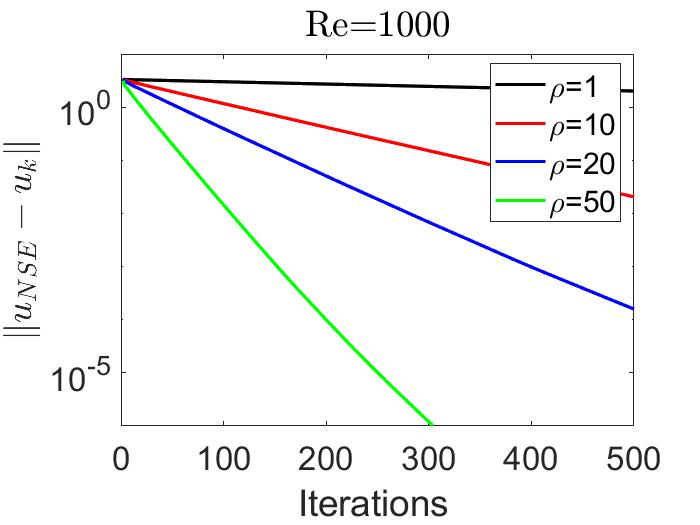}
	\\
\end{tabular}	

\caption{\label{fig:gamma0_varyingrho} Shown above is convergence behavior for the AH method applied to the driven cavity problem with varying $Re$ and $\rho$.  The top plots are without grad-div stabilization and the bottom plots are with it.  The first and third rows are for TH elements, and the second and last rows are SV.}
\end{figure}

\subsection{Lid-driven cavity}
We first test the AH method for steady NSE on the lid-driven cavity problem.  The domain for the problem is the unit square $\Omega=(0,1)^2$ and we impose Dirichlet boundary conditions by $ u |_{y=1}=(1,0)^T $ and $u=0$ everywhere else. We choose the parameter $\alpha = \nu^{-1}$.  We first illustrate how grad-div stabilization improves the AH method, and show the dramatic improvement offered by $(P_2,P^{disc}_1)$ Scott-Vogelius (SV) over $(P_2,P_1)$ Taylor-Hood (TH).  The second test shows even further dramatic improvement by incorporating AA.

\subsubsection{The effect of grad-div stabilization and comparison of Scott-Vogelius vs Taylor-Hood}

Our analysis above suggests that convergence of AH will be improved from using grad-div stabilization, and also from using SV  elements instead of TH since the connection to the iterated Picard penalty method is only made for SV elements.  Hence we now compare the AH method for both element choices, with and without grad-div stabilization (using parameter $\gamma=1$).  We run the tests for varying $\rho$ (to try to find a good choice of parameter $\rho$) and with varying $Re=\nu^{-1}$.  For these tests a uniform mesh with $h=1/32$ is used.

Results are shown in Figure \ref{fig:gamma0_varyingrho}.  We observe the best AH results clearly come from using SV instead of TH, and using $\gamma=1$ with SV gives by far the best results.  In most cases, TH fails to converge for any $\rho$.  Based on these results, we use SV elements for the rest of the numerical tests in this paper.

\subsubsection{Anderson accelerated grad-div stabilized AH method with SV elements}
We now consider the same test problem, using the AH iteration only with SV elements  and $\gamma=1$, but now adding AA.  We test AA depths $m=0$ (no acceleration),$  1, 5 $ and $ 10 $. Figure \ref{fig:AA100} and \ref{fig:AA1000} show convergence results obtained by Anderson accelerated grad-div stabilized AH method for $Re=100$ and $1000$, respectively, for varying $\rho$. As depth increases, the number of iterations decreases significantly. The fastest convergence is obtained with depths $ m = 50 $, but there is not much improvement past $m=5$.  

We also note that for optimally chosen $\rho$ in this setting (i.e. $\rho=20$ for $Re=100$ and $\rho=50$ for $Re=1000$ based on test above) with SV elements and grad-div stabilization, there is not much difference in convergence from AA.  However, for slightly non optimal $\rho$, there can be a dramatic improvement from AA.  Since one often does not know optimal $\rho$ a priori, the expected case in practice is using a non-optimal $\rho$.  

Comparing to existing literature, for $Re=100$ driven cavity it is reported in \cite{CHS17} that 731 iterations of AH were needed to converge to the same tolerance used herein and with Taylor-Hood elements, $\rho=1.2$ and $\alpha=70$.  Our results with Taylor-Hood elements and no grad-div stabilization were similarly bad, see figure \ref{fig:gamma0_varyingrho} in row 1; in fact, that they got convergence at all for this test is rather extraordinary.  With SV elements and grad-div stabilization, figure \ref{fig:gamma0_varyingrho} shows that with $\alpha=100$ and $\rho=20$, convergence is achieved in 80 iterations.  With less optimal parameter choices, AA can still keep the total iteration count low, see figure \ref{fig:AA100}.

\begin{figure}[h!]
	\centering
	\includegraphics[width = .3\textwidth, height=.24\textwidth]{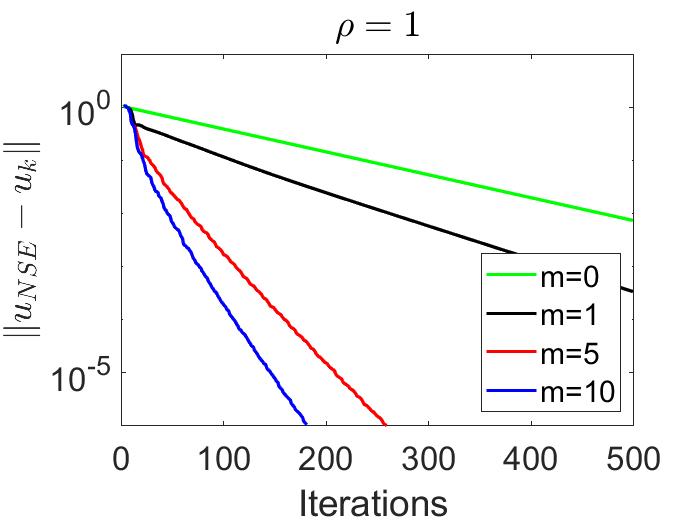}
		\includegraphics[width = .3\textwidth, height=.24\textwidth]{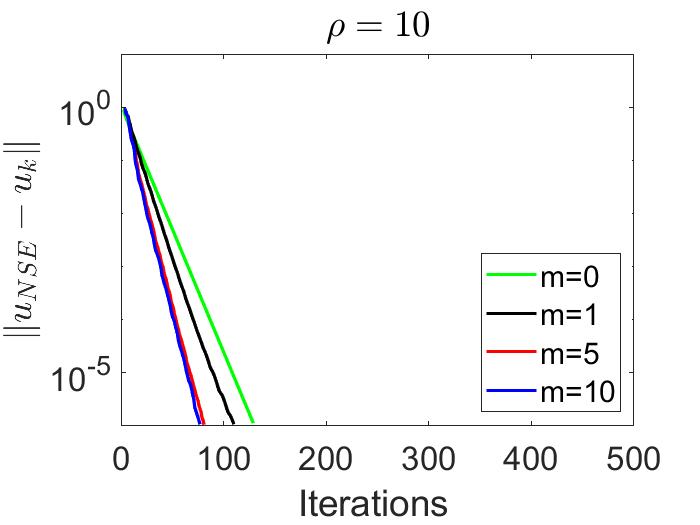}
			\includegraphics[width = .3\textwidth, height=.24\textwidth]{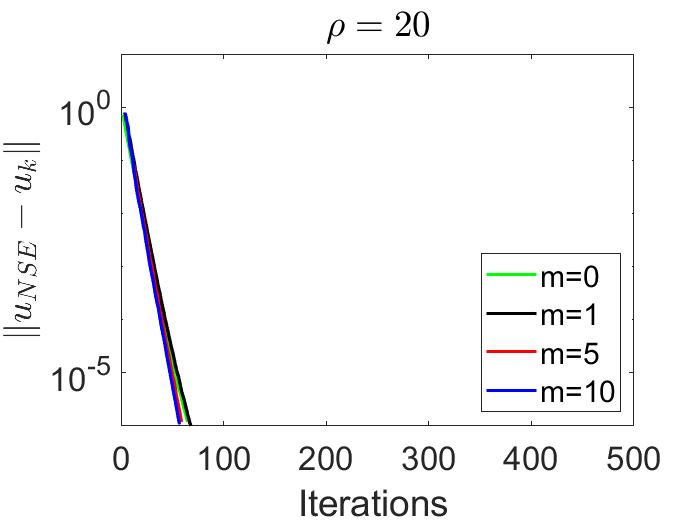}
	\caption{Convergence of Anderson accelerated grad-div stabilized Arrow-Hurwicz method for $Re=100$, different $\rho$'s, $\gamma=1$ with varying $m$. }\label{fig:AA100}
\end{figure}

\begin{figure}[H]
	\centering
	\includegraphics[width = .3\textwidth, height=.24\textwidth]{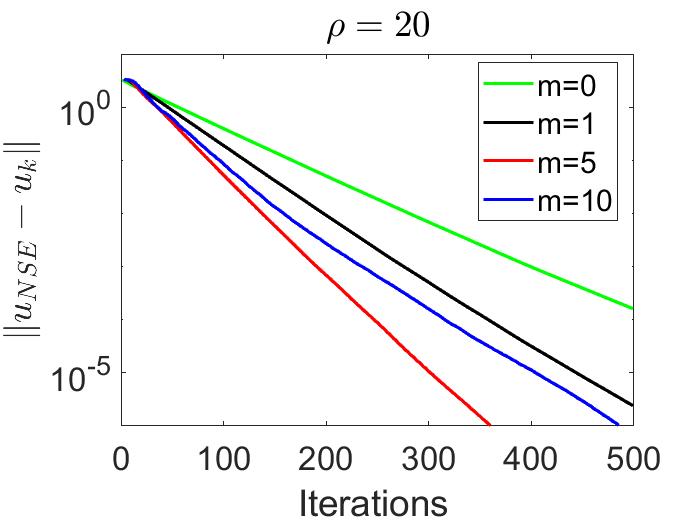}
	\includegraphics[width = .3\textwidth, height=.24\textwidth]{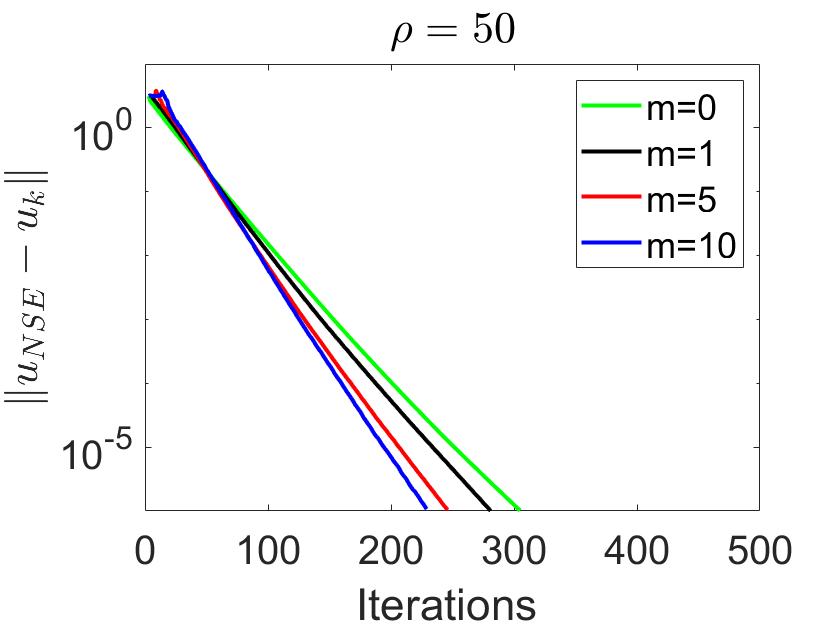}
	\includegraphics[width = .3\textwidth, height=.24\textwidth]{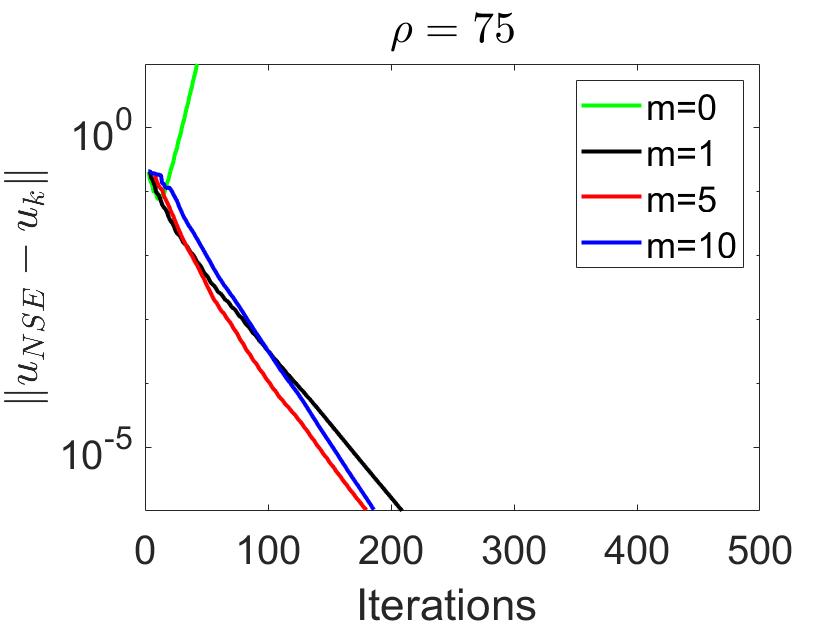}
	\caption{Convergence of Anderson accelerated grad-div stabilized Arrow-Hurwicz method for $Re=1000$, different $\rho$'s, $\gamma=1$  with varying $m$. }\label{fig:AA1000}
\end{figure}

\subsubsection{Driven cavity with $Re$=5,000 and $Re$=10,000}

As a final test with the driven cavity, we consider the case of $Re$=5,000 and $Re$=10,000 with $h$=1/64.  This is a difficult problem for nonlinear solvers \cite{PRX19,PR21}, and we show now that with the right parameter choices, the AH method can be effective for this problem.  For 5,000, $\rho=100, \ \gamma=1,\ m=100$ were used to obtain convergence in 464 iterations.  For 10,000, $\rho=150,\ \gamma=10,\ m=100$ was used to obtain convergence in 217 iterations.  Streamlines for both solutions are shown in Figure \ref{fig:streamlines}, and they are in good agreement with those from \cite{GGS82} even though we use a coarser mesh.  We note that, to date, the highest $Re$ for successful lid driven cavity computations in the literature is 1000 in \cite{CHS17}. 

\begin{figure}[H]
\centering
\includegraphics[width = .4\textwidth, height=.3\textwidth]{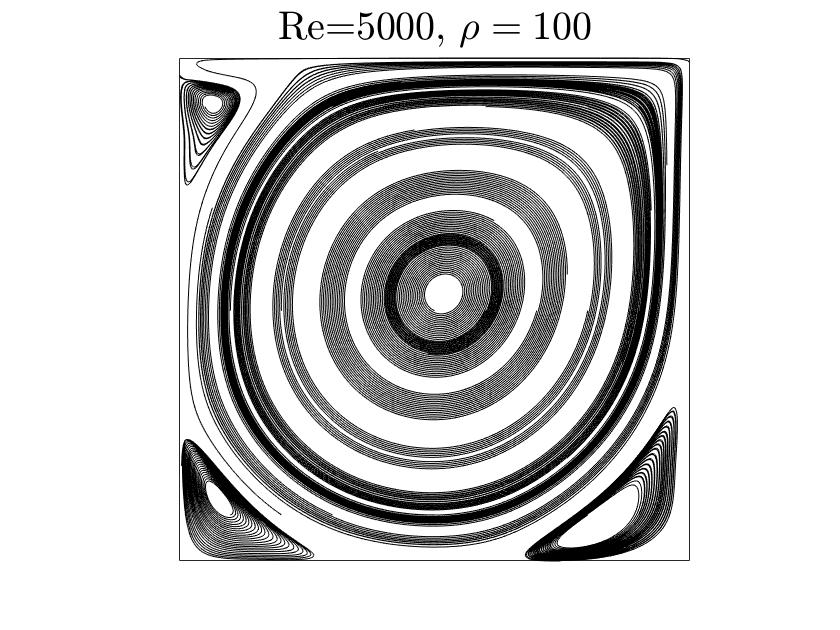}
\includegraphics[width = .4\textwidth, height=.3\textwidth]{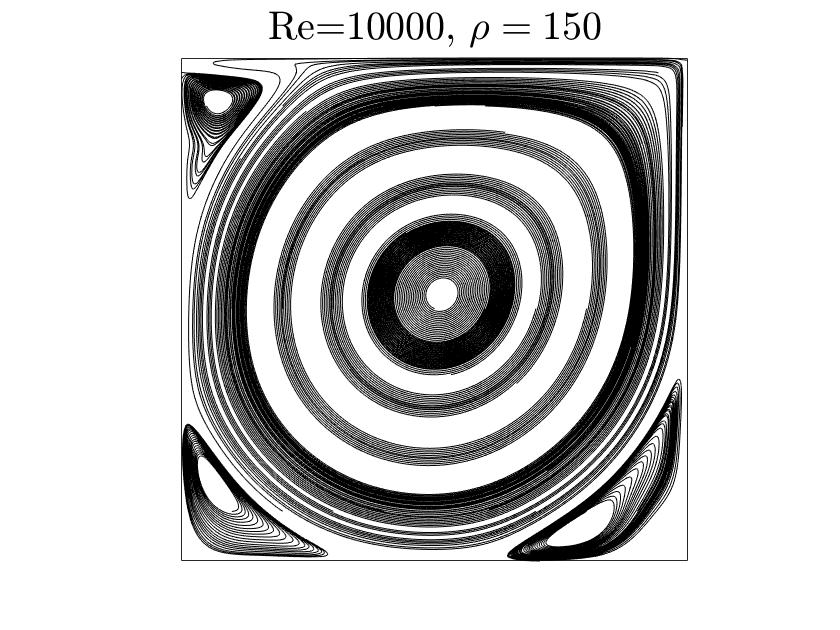}\caption{Show above are streamlines for $Re=5,000$ and $10,000$ driven cavity solutions. \label{fig:streamlines}}
\end{figure}
\subsection{Channel flow past a step}

\begin{figure}[H]
\centering
\includegraphics[width = .48\textwidth, height=.24\textwidth]{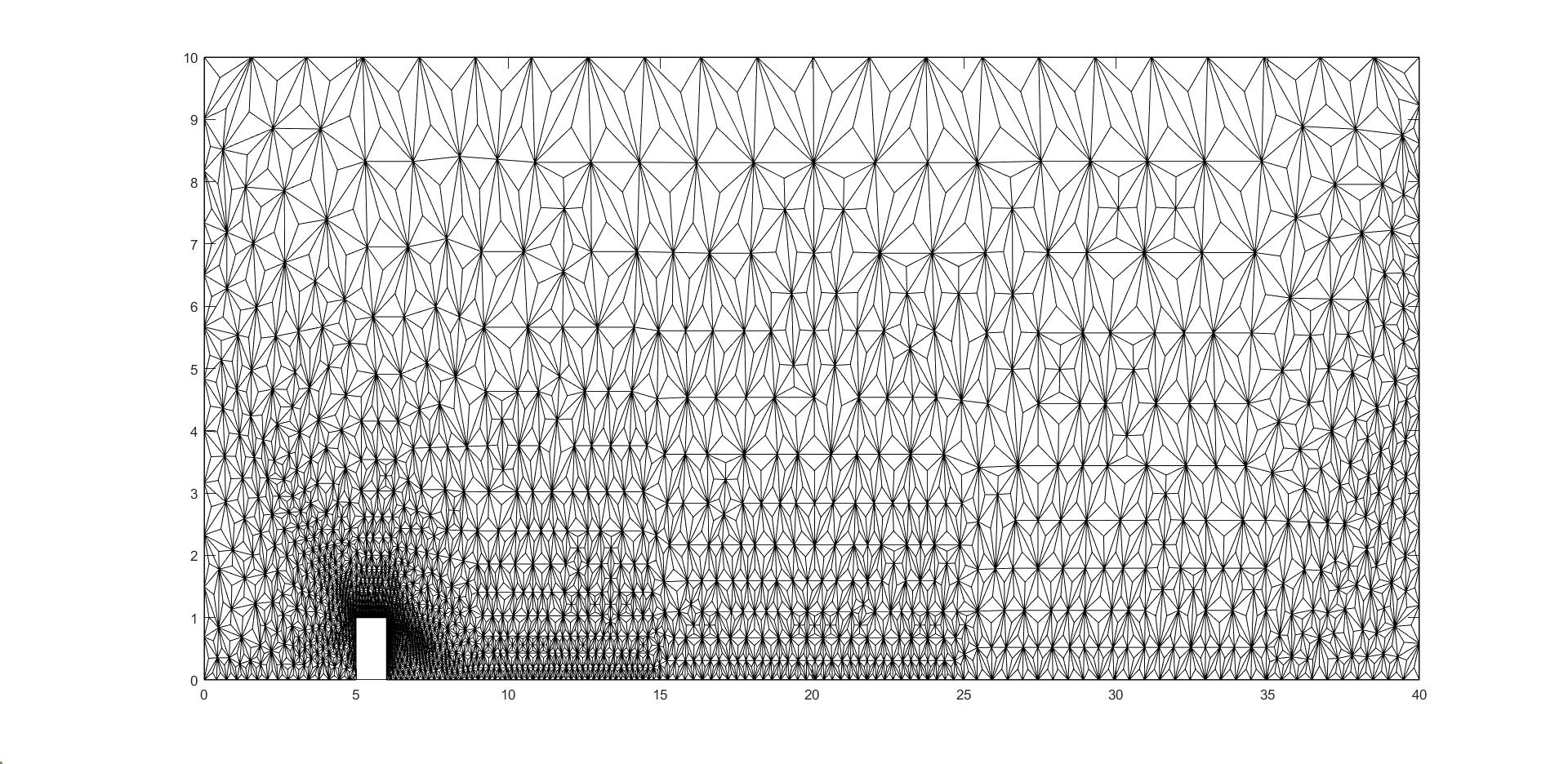}
\includegraphics[width = .48\textwidth, height=.24\textwidth]{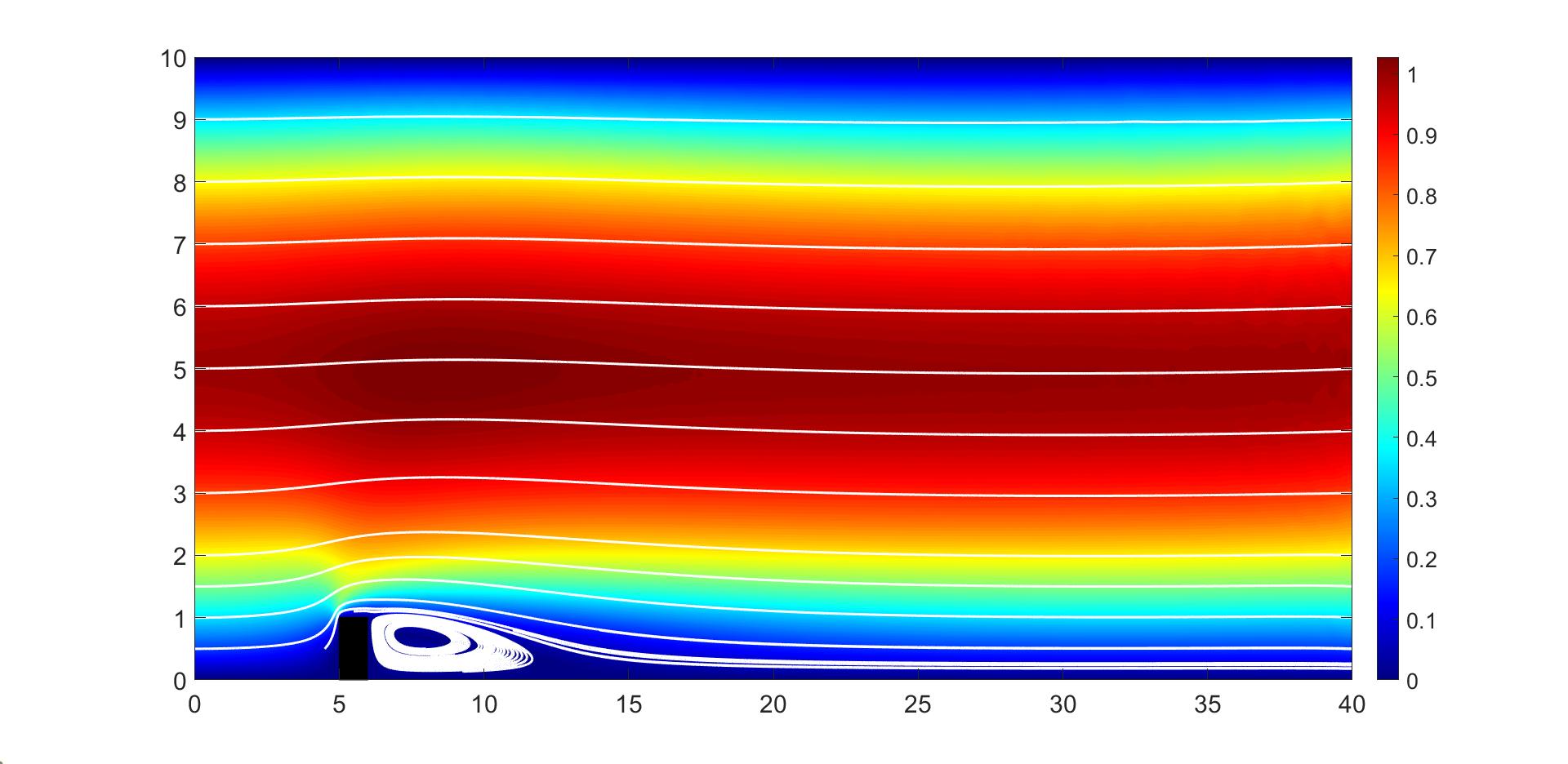}
\caption{Shown above are the mesh for channel flow past a step (left) and Re=100 velocity solution (right).  }\label{fig:mesh}
\end{figure}

For our last test, we consider 2D channel flow past a step with $Re=\nu^{-1}=100$.  The domain for this problem is a $40\times10 $
rectangular channel, with a $1\times1$ ‘step’ placed 5 units into the channel at the bottom. The triangulation we use is shown in Figure \ref{fig:mesh}, along with the $Re=100$ solution found with our solver (which is consistent with solutions from the literature \cite{JL06,GL81,G89}).  The discretization uses $(P_2,P_1^{disc})$ SV elements that provided 32,682 velocity degrees of freedom. 

First we consider $\gamma=10=\varepsilon^{-1}$, noting that obtaining convergence with $\gamma=1$ proved very difficult and we were not able to find a parameter set that gave convergence.  With $\gamma=10$, we computed four parameter sets: $(\rho=50,\ \alpha=\frac{\varepsilon}{\nu})$ , $(\rho=50,\ \alpha=\frac{1}{\nu})$, $(\rho=100,\ \alpha=\frac{\varepsilon}{\nu})$ - which is exactly the iterated penalty Picard method, and $(\rho=100,\ \alpha=\frac{1}{\nu})$.  Convergence plots for each of these parameter sets and varying $m$ are shown in figure \ref{fig:AHvsIPP}, and we observe that $m=100$ is the best choice for AA for all cases, and that AH with parameters chosen to match IPP performs significantly worse than other parameter choices.  The choice $\rho=50$ and $\alpha=\nu^{-1}$ with $m=100$ was very effective.  Results improve with $\gamma=100$ for all parameter sets, see figure \ref{fig:AHvsIPP100}.  Here, again $AH$ improves on IPP, with $\rho=50$ and $\rho=100$ withe $m=100$ were very effective parameter choices.

\begin{figure}[H]
\centering
\hspace{.2in} $(\rho=50,\ \alpha=\frac{\varepsilon}{\nu})$  \hspace{.35in} $(\rho=50,\ \alpha=\frac{1}{\nu})$ \hspace{.25in} $(\rho=100,\ \alpha=\frac{\varepsilon}{\nu})$ (IPP) \hspace{.15in} $(\rho=100,\ \alpha=\frac{1}{\nu})$ \hspace{.2in} \\

\includegraphics[width = .24\textwidth, height=.2\textwidth]{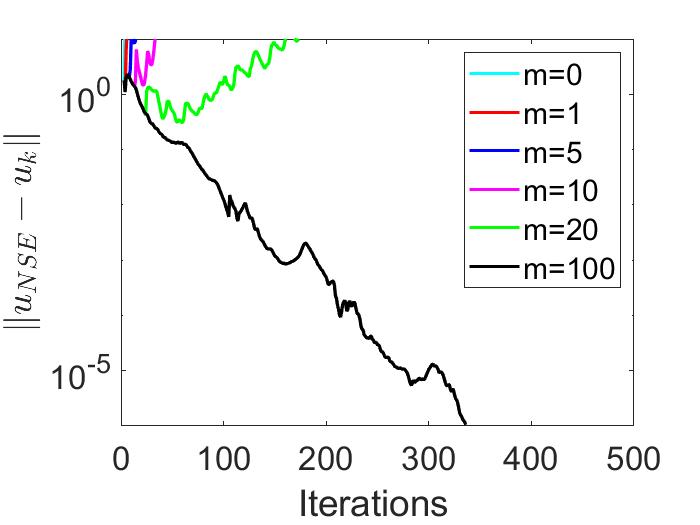}
\includegraphics[width = .24\textwidth, height=.2\textwidth]{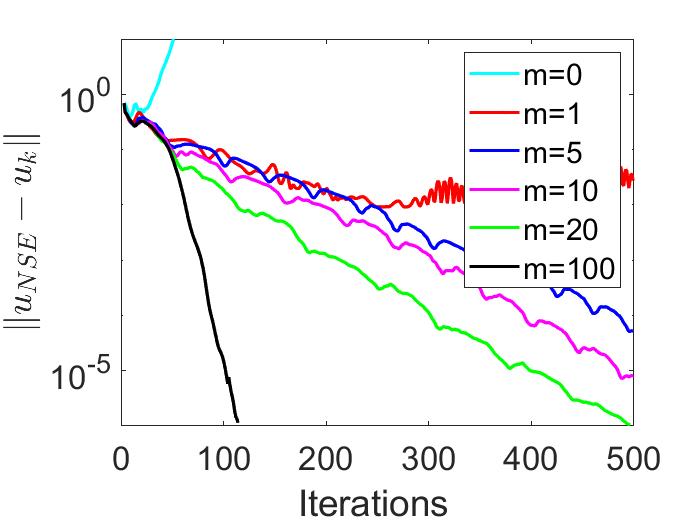}
\includegraphics[width = .24\textwidth, height=.2\textwidth]{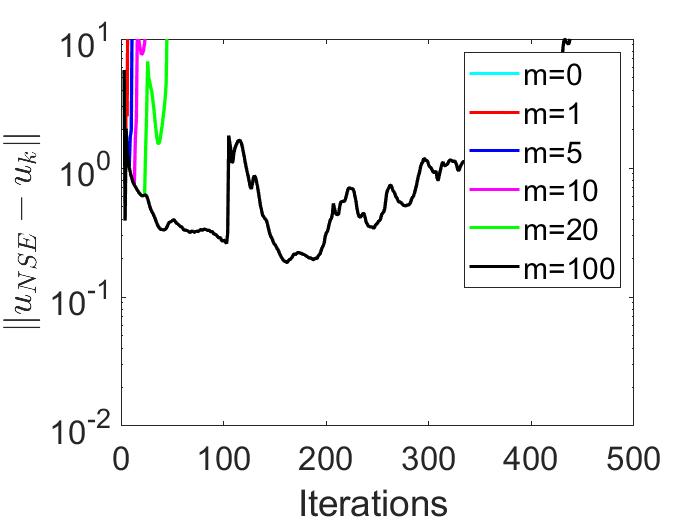}
\includegraphics[width = .24\textwidth, height=.2\textwidth]{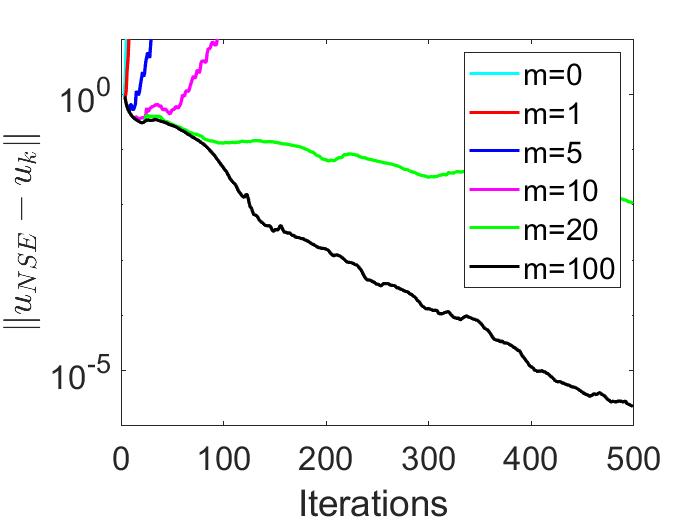}
\caption{Convergence of Anderson accelerated grad-div stabilized AH method and IPP iteration for $Re=100$ for varying parameters and $\gamma=10$.\label{fig:AHvsIPP}}
\end{figure}

\begin{figure}[H]
\centering
\hspace{.2in} $(\rho=50,\ \alpha=\frac{\varepsilon}{\nu})$  \hspace{.35in} $(\rho=50,\ \alpha=\frac{1}{\nu})$ \hspace{.25in} $(\rho=100,\ \alpha=\frac{\varepsilon}{\nu})$ (IPP) \hspace{.15in} $(\rho=100,\ \alpha=\frac{1}{\nu})$ \hspace{.2in} \\

\includegraphics[width = .24\textwidth, height=.2\textwidth]{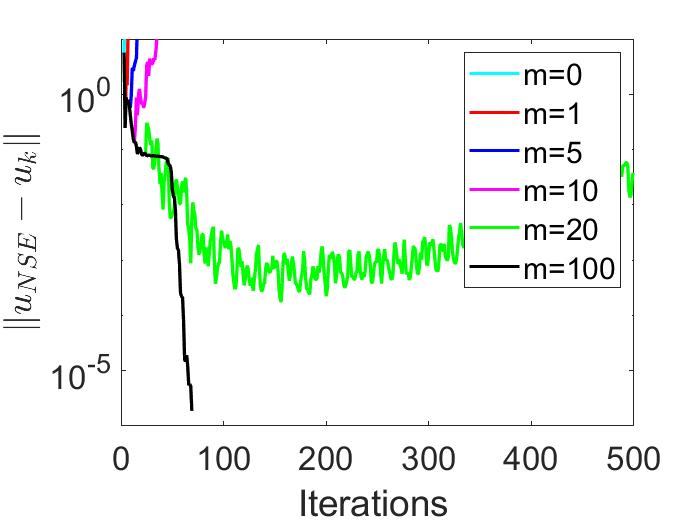}
\includegraphics[width = .24\textwidth, height=.2\textwidth]{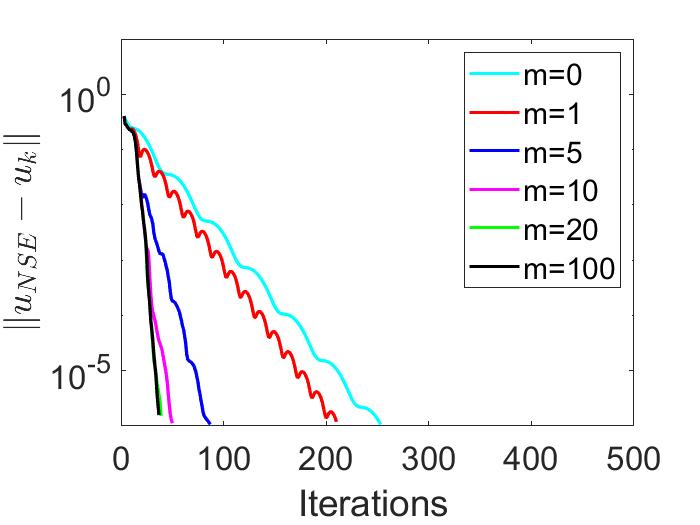}
\includegraphics[width = .24\textwidth, height=.2\textwidth]{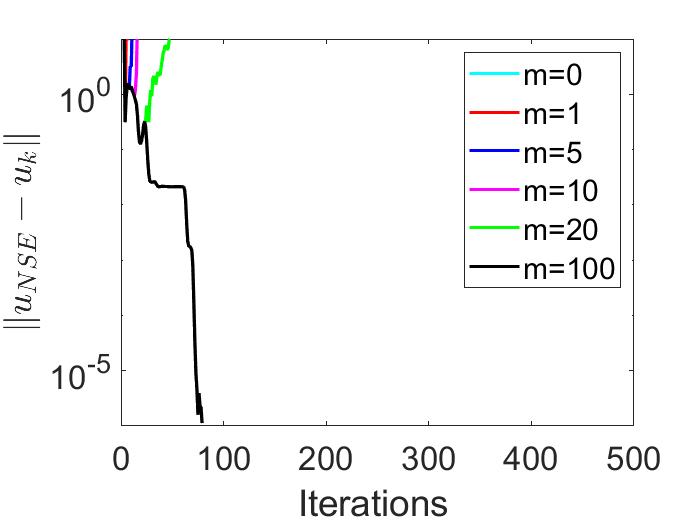}
\includegraphics[width = .24\textwidth, height=.2\textwidth]{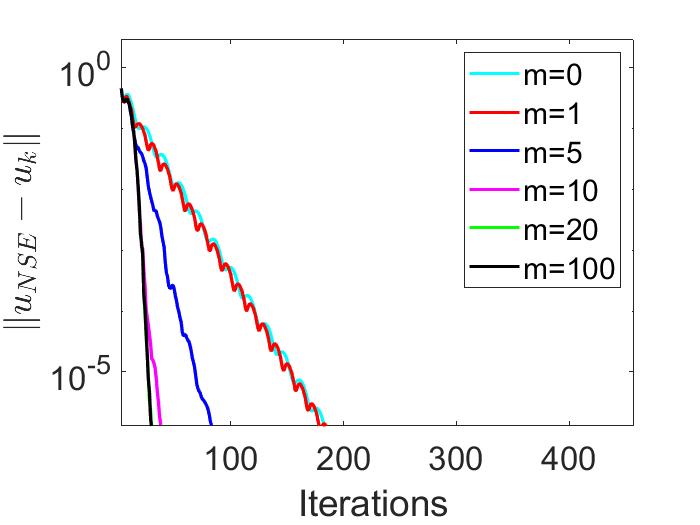}
\caption{Convergence of Anderson accelerated grad-div stabilized AH method and IPP iteration for $Re=100$ for varying parameters and $\gamma=100$.\label{fig:AHvsIPP100}}
\end{figure}

\section{Conclusions}

This paper developed multiple improvements to the AH method for solving the steady Navier-Stokes equations, and showed that with grad-div stabilization, SV elements and Anderson acceleration, the AH method can be a very effective and efficient solver.  SV elements and grad-div stabilization allowed us to connect AH to the well known iterated penalty Picard method, which has good convergence properties under small data \cite{RVX21}.  We also proved that the AH iteration, under certain conditions on the data and parameters, fits into the Anderson acceleration analysis framework developed in \cite{PR21} and thus AA improves the linear convergence rate of the AH method by the gain of the underlying AA optimization problem.  We also gave results of several numerical tests that show how each of these improvements is important for good convergence behavior, and when used together the AH method can be very effective.

\section{Acknowledgment}

\noindent{Author PG acknowledges partial support from National Science Foundation grant DMS 1907823.  Authors LR and DV acknowledge partial support from NSF grant DMS 2011490.}

\end{document}